\documentclass[11pt]{article}
\usepackage{amssymb}
\usepackage{amsmath}
\usepackage{amsfonts}

\newtheorem{theorem}{Theorem}

\newtheorem{corollary}[theorem]{Corollary}

\newtheorem{definition}[theorem]{Definition}

\newtheorem{lemma}[theorem]{Lemma}

\newtheorem{remark}[theorem]{Remark}

\newenvironment{proof}[1][Proof]{\noindent\textbf{#1.} }{\ \rule{0.5em}{0.5em}}
\input{tcilatex}

\begin{document}

\begin{center}
\bigskip {\huge On the Construction and Malliavin Differentiability of L\'{e}%
vy Noise Driven SDE's with Singular Coefficients}

\bigskip

{\Large Sven Haadem}$^{\text{1}}${\Large \ and Frank Proske}\footnote{%
Department of Mathematics, University of Oslo, Moltke Moes vei 35, Blindern,
P.O. Box 1053, Oslo, 0316, Norway,
\par
Email sven.haadem@cma.uio.no, proske@math.uio.no}

\bigskip

{\Large Abstract}
\end{center}

In this paper we introduce a new technique to construct unique strong
solutions of SDE's with singular coefficients driven by certain L\'{e}vy
processes. Our method which is based on Malliavin calculus does not rely on
a pathwise uniqueness argument. Furthermore, the approach, which provides a
direct construction principle, grants the additional insight that the
obtained solutions are Malliavin differentiable.

\bigskip


\bigskip

\section{\protect\bigskip Introduction}

Consider the stochastic differential equation (SDE) 
\begin{equation}
X_{t}=x+\int_{0}^{t}b(s,X_{s})ds+L_{t},0\leq t\leq T,x\in \mathbb{R}^{d},
\label{eq:main}
\end{equation}%
where $b:[0,T]\times \mathbb{R}^{d}\longrightarrow \mathbb{R}^{d}$ is a
Borel-measurable function and $L_{t},0\leq t\leq T$ is a $d-$dimensional
(square integrable) L\'{e}vy process, that is a process on some complete
probability space $\left( \Omega ,\mathcal{F},\mu \right) $ with stationary
and independent increments starting in zero (see e.g. \cite{Bertoin}).

\bigskip Using Picard iteration it is well known that there exists a unique
square integrable strong solution $X_{t},0\leq t\leq T$ \ to (\ref{eq:main})
if the drift coefficient $b$ is Lipschitz continuous and of linear growth.
Here, a strong solution to (\ref{eq:main}) means that $X_{t},0\leq t\leq T$
\ is an adapted process with respect to a $\mu -$completed filtration $%
\mathcal{F}_{t},0\leq t\leq T$ generated by $L_{t},0\leq t\leq T$ having c%
\`{a}dl\`{a}g paths and satisfying the equation (\ref{eq:main}) $\mu -$a.e.
See e.g. \cite{Protter}.

\bigskip In this are article, however, we are interested to study strong
solutions to (\ref{eq:main}) for certain L\'{e}vy processes, when $b$ is
singular in the sense that $b$ is bounded and $\alpha -$H\"{o}lder
continuous, i.e. 
\begin{equation*}
\left\Vert b\right\Vert _{C_{b}^{\alpha }}:=\sup_{0\leq t\leq T,x\in \mathbb{%
R}^{d}}\left\vert b(t,x)\right\vert +\sup_{0\leq t\leq T}\sup_{x\neq y}\frac{%
\left\vert b(t,x)-b(t,y)\right\vert }{\left\vert x-y\right\vert ^{\alpha }}%
<\infty .
\end{equation*}%
for some $0<\alpha <1$.

We mention that the analysis of strong solutions of SDE's with singular or
non-Lipschitz coefficients is important and has been of much current
interest for decades in stochastic analysis and its applications. Such
solutions naturally arise e.g. from a variety of applications in the theory
of controlled diffusion processes or in statistical mechanics to model
interacting infinite particle systems. See e.g. \cite{Kleptsyna}, \cite%
{Krylov1}, \cite{KR} and the references therein.

\bigskip The case, when $b$ is singular and $L_{t}$ is a Wiener process, has
been intensively studied in the litterature. A milestone in theory of SDE's
is a result due to A.K. Zvonkin, \cite{Zvonkin}, who constructed unique
strong solutions for Wiener process driven SDE's (\ref{eq:main}) on the real
line, when $b$ is merely bounded and measurable by employing estimates of
solutions of parabolic partial differential equations and a pathwise
uniqueness argument. Using similar techniques the latter result was
subsequently extended to the multidimensional case (\cite{Veretennikov}.
Further important generalizations of those results based on a pathwise
uniqueness argument can be e.g. found in \cite{KR}, \cite{GK} and \cite{GM}.
We also refer to \cite{DFPR}, where the authors use solutions to
infinite-dimensional Kolmogorov equations to prove strong uniqueness of
solutions to (\ref{eq:main}) for Wiener cylindrical processes $L_{t}$ on
Hilbert spaces, when $b$ is bounded and measurable. Another and more direct
approach to obtain strong solutions to (\ref{eq:main}) in the Wiener case,
which doesn't rely on a pathwise uniqueness argument and which is based on
techniques of Malliavin calculus, was studied in \cite{MP},\cite{MMNPZ}. See
also \cite{FNP} in the case of Hilbert spaces.

If the the driving process $L_{t}$ in (\ref{eq:main}), however, is a pure
jump L\'{e}vy process we observe major differences to the Gaussian case. For
example, if $L_{t}$, $0\leq t\leq T$ is a one-dimensional symmetric $\alpha
- $stable process for $0<\alpha <1$ then one can find a bounded $\gamma -$H%
\"{o}lder-continuous drift coefficient $b$ with $\alpha +\gamma <1$ such
that pathwise uniqueness of solutions to (\ref{eq:main}) fails. See \cite%
{TTW}. Similar results on non-pathwise uniqueness of solutions of SDE's with
multiplicative symmetric $\alpha $-stable noise were obtained by \cite{BBZ}.
See also \cite{Bass04}, \cite{Priola12} and the references therein. As for
the study of weak solutions of SDE's driven by L\'{e}vy processes we shall
refer here e.g. to \cite{Bass88}, \cite{Zanzotto02} and \cite{PZ}. Further,
martingale problems of SDE's driven by symmetric $\alpha -$stable processes
were treated in \cite{BZ}.

In this paper we aim at introducing a new technique to construct (unique)
strong solutions to (\ref{eq:main}). We illustrate this principle, which can
be also applied to a variety of other L\'{e}vy processes, by considering the
special case of a truncated $\alpha -$stable process of index $\alpha \in
(1,2)$. Our method differs from the above mentioned ones in the sense that
we do not resort to the Yamada-Watanabe principle to guarantee strong
uniqueness of solutions, that is we do not require pathwise uniqueness in
connection with the existence of a weak solution to find a unique strong
solution to (\ref{eq:main}). In fact our approach, which provides a direct
construction of strong solutions, can be regarded as a synthesis of
techniques developed in \cite{MP}, \cite{MMNPZ} and \cite{FGP} (or \cite%
{Priola12} in the case of symmetric $\alpha -$stable processes) applied to L%
\'{e}vy processes. More precisely, we approximate the singular coefficient $%
b $ in (\ref{eq:main}) by smooth functions $b_{n}$ admitting a unique strong
solution $X_{t}^{n}$ 
\begin{equation}
X_{t}^{n}=x+\int_{0}^{t}b_{n}(s,X_{s}^{n})ds+L_{t},0\leq t\leq T,x\in 
\mathbb{R}^{d}
\end{equation}%
for each $n\geq 1$. Then we recast the integral $%
\int_{0}^{t}b_{n}(s,X_{s}^{n})ds$ in (\ref{eq:main}) by using solutions to a
backward Kolmogorov equation associated with $L_{t}$ in terms of a more
regular expression (see \cite{FGP}, \cite{Priola12}). Finally, we apply a
new compactness criterion of square integrable functionals of L\'{e}vy
processes based on Malliavin calculus to the sequence of solutions $%
X_{t}^{n},n\geq 1$ to obtain a unique strong solution $X_{t}$ (compare \cite%
{MP}, \cite{MMNPZ} in the Wiener process case). Moreover, our method gives
the crucial additional insight that $X_{t}$ is Malliavin differentiable for
all $t$. See \cite{Nua06} or \cite{NOP09} for more information on Malliavin
calculus.

\bigskip

Our paper is organized as follows: In Section 3 we introduce some notation
and recall some basic results from the theory of L\'{e}vy processes and
Malliavin calculus which we will use throughout the article. In Section 3.1
we prove a new compactness criterion for square integrable functionals of L%
\'{e}vy processes and establish certain estimates of solutions of Kolmogorov
type equations associated with L\'{e}vy processes. Finally, in Section 4 we
apply the results of the previous section to prove our main result on the
existence of a unique and Malliavin differentiable strong solution to (\ref%
{eq:main}) for certain L\'{e}vy processes (Theorem \ref{MainResult}).

\section{Framework}

In this section we briefly introduce the mathematical framework we want to
apply in the subsequent sections.

\subsection{H\"{o}lder Spaces}

For $\beta \in (0,1)$ and $k,d\geq 1$, denote by $C_{b}^{\beta }(\mathbb{R}%
^{d},\mathbb{R}^{k})$ the space of bounded $\beta -$H\"{o}lder continuous
functions, that is the space of continuous functions $u:$ $\mathbb{R}%
^{d}\longrightarrow \mathbb{R}^{k}$ such that 
\begin{equation*}
\lVert u\rVert _{C_{b}^{\beta }(\mathbb{R}^{d},\mathbb{R}^{k})}:=\lVert
u\rVert _{\infty }+\underset{x\neq y}{\sup }\frac{\lvert u(x)-u(y)\rvert }{%
\lvert x-y\rvert ^{\beta }}<\infty ,
\end{equation*}%
where $\lVert u\rVert _{\infty }:=\sup_{x\in \mathbb{R}^{d}}\lvert
u(x)\rvert $and $\left\vert \cdot \right\vert $ is the Euclidean norm. We
also simply write $C_{b}^{\beta }(\mathbb{R}^{d})=C_{b}^{\beta }(\mathbb{R}%
^{d},\mathbb{R}).$ Further, we denote by $C_{b}^{i,\beta }(\mathbb{R}^{d})$
for $i\geq 1$ and $0<\alpha <1$ the Banach space of all $i$-times Fr\'{e}%
chet differentiable functions $u:\mathbb{R}^{d}\longrightarrow \mathbb{R}$
with $D^{l}u\in C_{b}^{\beta }(\mathbb{R}^{d},(\mathbb{R}^{d})^{\otimes
(l+1)}),l=1,...,i$ and norm $\left\Vert \cdot \right\Vert _{C_{b}^{i,\beta }(%
\mathbb{R}^{d})}$ given by%
\begin{equation*}
\left\Vert u\right\Vert _{C_{b}^{i,\beta }(\mathbb{R}^{d})}:=\lVert u\rVert
_{\infty }+\sum_{l=1}^{i}\lVert D^{l}u\rVert _{\infty }+\underset{x\neq y}{%
\sup }\frac{\lvert D^{i}u(x)-D^{i}u(y)\rvert }{\lvert x-y\rvert ^{\beta }}.
\end{equation*}%
We let $C_{b}^{0,\beta }(\mathbb{R}^{d}):=C_{b}^{\beta }(\mathbb{R}^{d}).$
For notational convenience we also denote the norm of the Banach space $%
C([0,T],C_{b}^{i,\beta }(\mathbb{R}^{d}))$ by $\left\Vert \cdot \right\Vert
_{C_{b}^{i,\beta }}$ defined as%
\begin{equation*}
\left\Vert u\right\Vert _{C_{b}^{i,\beta }}=\sup_{0\leq t\leq T}\left\Vert
u(t,\cdot )\right\Vert _{C_{b}^{i,\beta }(\mathbb{R}^{d})}.
\end{equation*}

\bigskip

\subsection{L\'{e}vy Processes}

We give a concise summary of basic facts of the theory of L\'{e}vy
processes. The reader may consult \cite{Bertoin} or \cite{NOP09} for further
information.\newline

Given a complete probability space, $(\Omega ,\mathcal{F},P)$, we a L\'{e}vy
process is defined as follows.

\begin{definition}
A stochastic process $L(t)\in \mathbb{R}^{d}$, $t\geq 0$ is called a L\'{e}%
vy process if the following properties hold:

\begin{enumerate}
\item $L(0)=0$ $P$-a.s.,

\item the process has independent increments, that is, for all $t>0$ and $%
h>0 $, the increment $L(t+h)-L(h)$ is independent of $L(s)$ for all $s\leq t$%
,

\item the process has stationary increments, that is, for all $h>0$, the
increment $L(t+h)-L(h)$ has the same law as $L(h)$,

\item the process is stochastically continuous, that is, for every $t>0$ an $%
\epsilon >0$ we have that $\lim_{s\rightarrow t}P\{\lvert L(t)-L(s)\rvert
>\epsilon \}=0$,

\item the paths of the process are c\`adl\`ag, that is, the trajectories are%
\newline
right-continuous with existing left limits.
\end{enumerate}
\end{definition}

\bigskip

Now, define the jump of $L$ at time $t$ as 
\begin{equation*}
\Delta L(t):=L(t)-L(t^{-}).
\end{equation*}%
Let $\mathbb{R}_{0}^{d}:=\mathbb{R}^{d}\setminus \{0\}$ and let $\mathcal{B}(%
\mathbb{R}_{0}^{d})$ be the Borel-$\sigma $-algebra on $\mathbb{R}_{0}^{d}$.
Further, we now introduce a Poisson random measure on $\mathcal{B}([0,\infty
))\times \mathcal{B}(\mathbb{R}_{0}^{d})$ by 
\begin{equation*}
N(t,U):=\sum_{0\leq s\leq t}\mathbf{1}_{U}(\Delta L(s))
\end{equation*}%
for $U\in \mathcal{B}(\mathbb{R}_{0}^{d})$. This is the jump measure of $%
\eta $. The \emph{L\'{e}vy measure} $\nu $ of $\eta $ is defined by 
\begin{equation*}
\nu (U):=E[N(1,U)],
\end{equation*}%
for $U\in \mathcal{B}(\mathbb{R}_{0}^{d}).$

\bigskip It can be shown that the characteristic function of a L\'{e}vy
process is given by the following L\'{e}vy-Khintchine formula (see e.g. \cite%
{Bertoin}):%
\begin{equation}
E[\exp (i\left\langle L(t),u\right\rangle )]=\exp (-t\Psi (u)),u\in \mathbb{R%
}^{d},t\geq 0,  \label{Khintchine}
\end{equation}%
where $\Psi $ is the characteristic exponent%
\begin{equation*}
\Psi (u)=-\int_{\mathbb{R}^{d}}(e^{i\left\langle u,y\right\rangle
}-1-i\left\langle u,y\right\rangle \mathbf{1}_{\left\{ \left\vert
y\right\vert \leq 1\right\} })\nu (dy).
\end{equation*}

\bigskip

Let us define the compensated jump measure $\widetilde{N}$ by 
\begin{equation*}
\tilde{N}(ds,dz):=N(ds,dz)-\nu (dz)dt.
\end{equation*}

It turns out that L\'{e}vy processes have the following representation:

\begin{theorem}[The L\'{e}vy-It\^{o} decomposition]
Let $L$ be a L\'{e}vy process. Then $L$ admits the following integral
representation 
\begin{equation*}
\eta (t)=at+\sigma W(t)+\int_{0}^{t}\int_{|z|<1}z\tilde{N}%
(ds,dz)+\int_{0}^{t}\int_{|z|>1}zN(ds,dz)
\end{equation*}%
for some $a\in \mathbb{R}^{d},\sigma \in \mathbb{R}^{d\times d}\mathbb{\ }$%
and a standard Wiener process $W(t),t\geq 0$.
\end{theorem}

\bigskip

Let us recall the infinitesimal generator $\mathcal{L}$ of the L\'{e}vy
processes $L_{t},t\geq 0$:\newline
The infinitesimal generator of $L_{t},t\geq 0$ is the operator $\mathcal{L}$%
, which is defined to act on suitable functions $f$ of some Banach space
such that 
\begin{equation*}
\mathcal{L}f(x)=\lim_{t\rightarrow 0^{+}}\frac{E^{x}[f(L_{t})]-f(x)}{t}
\end{equation*}%
exists.

\subsection{Chaos Expansions and the Malliavin Derivative}

In this subsection we briefly recall the concept of the Malliavin derivative
with respect to L\'{e}vy processes as a central notion of Malliavin
calculus. We refer the reader to the books \cite{Nua06} and \cite{NOP09} for
more information on Malliavin calculus.

For notational convenience, we assume in this subsection $d=1$. Consider $%
\Omega =\mathcal{S}(\mathbb{R}^{d})=\mathcal{S}(\mathbb{R})$, the space of
tempered distributions on $\mathbb{R}$. Then we know from the
Bochner-Minlos-Sazonov theorem (see e.g. \cite{Xiong}) that there exists a
probablility measure $\mu $, such that 
\begin{equation*}
\int_{\Omega }e^{i\langle \omega ,f\rangle }\mu (d\omega )=exp(\int_{\mathbb{%
R}}\Psi (f(x))dx),
\end{equation*}%
for $f\in \mathcal{S}(\mathbb{R})$, where $\Psi $ is the characteristic
exponent given by 
\begin{equation*}
\Psi (u)=\int_{\mathbb{R}}(e^{iuz}-1-iuz)\nu (dz),
\end{equation*}%
where $<\omega ,f>$ denotes the action of $\omega \in \mathcal{S}^{^{\prime
}}(\mathbb{R})$ (Schwartz distribution space) on $f\in \mathcal{S}(\mathbb{R}%
)$ and where $\nu $ is a L\'{e}vy measure. The triple $(\Omega ,\mathcal{F}%
,\mu )$ is called the (pure jump) L\'{e}vy white noise probability space.

From now on we assume a square integrable L\'{e}vy process $L_{t},t\geq 0$
with L\'{e}vy measure $\nu $ constructed on $(\Omega ,\mathcal{F},\mu )$.

In what follows we want to use the chaos representation property of a square
integrable L\'{e}vy process to define the Malliavin derivative with respect
to such processes. To this end we need some notation:

Let us denote by $\mathcal{I}$ the set of all finite multi-indices $\alpha
=(\alpha _{1},\alpha _{2},\ldots ,\alpha _{m})$, $m\in \mathbb{N}_{0}$ of
non-negative integers $a_{i}$, $i=1,\ldots ,m$, and define $\lvert \alpha
\rvert :=\alpha _{1}+\ldots \alpha _{m}$. Further, let $e_{i}$, $i\geq 1$ be
an orthonormal basis of $L^{2}(\lambda \times \nu )$ ($\lambda $ Lebesgue
measure) and let for $\alpha =(\alpha _{1},\alpha _{2},\ldots ,\alpha
_{m})\in \mathcal{I}$ 
\begin{align*}
H_{\alpha }& =\int_{\mathbb{R}}\int_{\mathbb{R}_{0}}\cdots \int_{\mathbb{R}%
}\int_{\mathbb{R}_{0}}e_{1}^{\otimes \alpha _{1}}\hat{\otimes}\cdots \hat{%
\otimes}e_{m}^{\otimes \alpha _{m}}((s_{1},z_{1}),\ldots ,(s_{m},z_{m})) \\
& \widetilde{N}(ds_{1},dz_{1})\cdots \widetilde{N}(ds_{m},dz_{m}),
\end{align*}%
where $\otimes $ and $\hat{\otimes}$ denotes the tensor product and the
symmetrized tensor product, respectively.

Then $\{H_{\alpha }:\alpha \in \mathcal{I}\}$ forms an orthogonal basis of $%
L^{2}(\mu )$:

\begin{theorem}[Chaos expansion]
Any $X \in L^2(\mu)$ has the unique chaos decomposition of the form 
\begin{align}  \label{eq:decomp}
X = \sum_{\alpha \in \mathcal{I}} c_{\alpha} H_{\alpha}
\end{align}
with $c_{\alpha} \in \mathbb{R}$. Moreover 
\begin{align*}
\lVert X \rVert^2_{L^2(\mu)} = \sum_{\alpha \in \mathcal{I}} \alpha!
c_{\alpha},
\end{align*}
where 
\begin{align*}
\alpha ! := \alpha_1!,\alpha_2!,\ldots,\alpha_m!
\end{align*}
for $\alpha= (\alpha_1, \alpha_2,\ldots,\alpha_m)$.
\end{theorem}

\bigskip

We are now ready to define the Malliavin derivative.

We define the Malliavin derivative of a square integrable functional $X$ of
a pure jump L\'{e}vy process $L$ with chaos expansion 
\begin{equation}
X=\sum_{\alpha \in \mathcal{I}}c_{\alpha }H_{\alpha }  \label{ChaosExpansion}
\end{equation}%
by 
\begin{equation*}
D_{t,z}X=\sum_{\beta \in \mathcal{I}}\sum_{i\in \mathbb{N}}(c_{\beta
+\epsilon _{i}}(\beta _{i}+1))e_{i}(t,z)H_{\beta },
\end{equation*}%
provided $X$ belongs to the domain $\mathbb{D}^{1,2}\subset L^{2}(\mu )$
given by%
\begin{eqnarray*}
\mathbb{D}^{1,2} &:&=\left\{ X\in L^{2}(\mu )\text{ with chaos expansion (%
\ref{ChaosExpansion}): }\right. \\
&&\left. \sum_{\beta \in \mathcal{I}}\sum_{i\in \mathbb{N}}(c_{\beta
+\epsilon _{i}}(\beta _{i}+1))^{2}\beta !<\infty \right\} ,
\end{eqnarray*}%
where $\epsilon _{i}:=(0,0,...,0,1,0,...,0)$ with $1$ in the $i-$th position.

\bigskip

\subsection{Fractional Sobolev Spaces}

In this paper we aim at constructing strong solutions to L\'{e}vy noise
driven SDE's by using Banach spaces of functions related to fractional
Sobolev spaces (or Sobolev-Slobodeckij spaces). See \cite{AF03} for more
information about these spaces.

\begin{definition}
Let $0<\alpha <2,1\leq p<\infty $ and let $\Omega \subset \mathbb{R}^{d}$ be
a Lipschitz domain. Then, the fractional Sobolev space $W^{\alpha ,p}(\Omega
)$ can be defined as%
\begin{eqnarray*}
&&W^{\alpha ,p}(\Omega ) \\
&=&\left\{ f:\Omega \longrightarrow \mathbb{R}:\left\Vert f\right\Vert
_{W^{\alpha ,p}(\Omega )}:=\right. \\
&&\left. (\left\Vert f\right\Vert _{L^{p}(\Omega )}^{p}+\int_{\Omega
}\int_{\Omega }\frac{\left\vert f(x)-f(y)\right\vert ^{p}}{\left\vert
x-y\right\vert ^{d+2\alpha }}dxdy)^{1/p}<\infty \right\} .
\end{eqnarray*}
Here, 
\begin{equation*}
\left[ f\right] :=(\int_{\Omega }\int_{\Omega }\frac{\left\vert
f(x)-f(y)\right\vert ^{p}}{\left\vert x-y\right\vert ^{d+2\alpha }}%
dxdy)^{1/p}
\end{equation*}%
denotes the Slobodeckij semi-norm.
\end{definition}

\bigskip The Sobolev-Slobodeckij spaces form a scale of Banach spaces, i.e.
one has the continuous injections or embeddings 
\begin{equation*}
W^{k+1,p}(\Omega )\hookrightarrow W^{s^{\prime },p}(\Omega )\hookrightarrow
W^{s,p}(\Omega )\hookrightarrow W^{k,p}(\Omega ),\quad k\leq s\leq s^{\prime
}\leq k+1.
\end{equation*}%
Sobolev-Slobodeckij spaces are special cases of Besov spaces. See e.g. \cite%
{AF03}.

Another approach to define fractional order Sobolev spaces $W^{\alpha
,p}(\Omega )$ is

\begin{definition}
\begin{equation*}
W^{\alpha ,p}(\Omega ):=\{f\in L^{p}(\Omega ):\mathcal{F}^{-1}(1+|\xi
|^{2})^{\frac{\alpha }{2}}\mathcal{F}f\in L^{p}(\Omega )\}
\end{equation*}%
with the norm 
\begin{equation*}
\Vert f\Vert _{W^{k,p}}:=\Vert \mathcal{F}^{-1}(1+|\xi |^{2})^{\frac{k}{2}}%
\mathcal{F}f\Vert _{L^{p}},
\end{equation*}%
where $\mathcal{F}$ denotes the Fourier-transform. This space is also called
a Bessel potential space. $\Omega $ is a domain with uniform $C^{k}$%
-boundary, $k$ a natural number and $1<p<\infty $.
\end{definition}

By the embeddings 
\begin{equation*}
W^{k+1,p}(\mathbb{R}^{n})\hookrightarrow W^{s^{\prime },p}(\mathbb{R}%
^{n})\hookrightarrow W^{s,p}(\mathbb{R}^{n})\hookrightarrow W^{k,p}(\mathbb{R%
}^{n}),\quad k\leq s\leq s^{\prime }\leq k+1
\end{equation*}%
the Bessel potential spaces form a continuous scale between these Sobolev
spaces.

\section{\protect\bigskip Preliminary Results}

\QTP{bTë?}
In this section we give a new compactness criterion for square integrable
functionals (of pure jump) L\'{e}vy processes based on Malliavin calculus.
Further, we prove some regularity results of solutions of Kolmogorov type
equations associated with certain L\'{e}vy processes. We aim at employing
these results in Section 4 to establish our main results on the existence
and uniqueness of Malliavin differentiable strong solutions to SDEs of the
form (\ref{eq:main}).

\subsection{Compactness Criterion}

\QTP{bTë?}
Our construction method of solutions to (\ref{eq:main}) requires a
compactness criterion for subsets of $L^{2}(\mu )$. So we prove the
following theorem which can be regarded as an extension of \cite{PMN} from
Wiener processes to (pure jump) L\'{e}vy processes.

\begin{theorem}[Compactness in $L^{2}(\protect\mu )$]
\label{Theorem 3} Let $C$ be a selfadjoint compact operator on $H\otimes
L^{2}(\nu )$ with dense image, where $H:=L^{2}([0,1])$. Then for any $c>0$
the set 
\begin{equation*}
\mathcal{G}=\{G\in \mathbb{D}^{1,2}:\lVert G\rVert _{L^{2}(\Omega )}+\lVert
C^{-1}DG\rVert _{L^{2}(\Omega ;H\otimes L^{2}(\nu ))}\leq c\}
\end{equation*}%
is relatively compact in $L^{2}(\mu )$.
\end{theorem}

\begin{proof}
The proof is similar to that of Theorem 1 in \cite{PMN}. Consider a complete
orthonormal system $\{e_{i}\}_{i\geq 1}$ of $H\otimes L^{2}(\nu )$. Assume
that $Ce_{i}=\beta _{i}e_{i}$ with $\beta _{i}>0$ for all $i\geq 1$. Note
that the compactness of $C$ implies that $\lim_{i\rightarrow \infty }\beta
_{i}=0$. Let $G\in \mathbb{D}^{1,2}$ be a random variable such that 
\begin{equation*}
\lVert G\rVert _{L^{2}(\Omega )}+\lVert C^{-1}DG\rVert _{L^{2}(\Omega
;H\otimes L^{2}(\nu ))}\leq c.
\end{equation*}%
Let 
\begin{equation*}
G=\sum_{\gamma \in \mathcal{I}}c_{\gamma }H_{\gamma }
\end{equation*}%
be the chaos decomposition of $G$. Then 
\begin{equation*}
D_{\cdot ,\cdot }G=\sum_{\gamma \in \mathcal{I}}\left( \sum_{k}c_{\gamma
+\epsilon ^{k}}(\gamma _{k}+1)e_{k}(\cdot ,\cdot )\right) H_{\gamma },
\end{equation*}%
where $\epsilon ^{k}\in \mathcal{I}$ is defined by 
\begin{equation*}
\epsilon ^{j}=%
\begin{cases}
1, & \text{ if }\epsilon _{j}^{j}=1 \\ 
0, & \text{otherwise.}%
\end{cases}%
\end{equation*}%
See Section 2.3. From this we get that 
\begin{equation*}
\lVert G\rVert _{L^{2}(\Omega )}^{2}=\sum_{\alpha }\alpha !c_{\alpha }^{2}
\end{equation*}%
and 
\begin{align*}
\lVert C^{-1}D_{\cdot ,\cdot }G\rVert _{L^{2}(\Omega ;H\otimes L^{2}(\nu
))}^{2}& =\sum_{\gamma }\gamma !\sum_{k}(\gamma _{k}+1)^{2}\frac{1}{\beta
_{k}^{2}}c_{\gamma +\epsilon ^{k}}^{2} \\
& =\sum_{\gamma }\sum_{k}(\gamma -\epsilon ^{k})!c_{\gamma }^{2}\frac{1}{%
\beta _{k}^{2}}\gamma _{k}^{2} \\
& =\sum_{\gamma }c_{\gamma }^{2}\gamma !\sum_{k}\frac{1}{\beta _{k}^{2}}%
\gamma _{k}^{2}\frac{(\gamma -\epsilon ^{k})!}{\gamma !} \\
& =\sum_{\gamma }c_{\gamma }^{2}\gamma !\sum_{k}\frac{1}{\beta _{k}^{2}}%
\gamma _{k}^{2}.
\end{align*}%
For fixed $R>0$ define the set 
\begin{equation*}
A_{R}=\left\{ \alpha \in \mathcal{I}:\sum_{k}\frac{1}{\beta _{k}^{2}}\alpha
_{k}<R\right\} .
\end{equation*}%
Since $\underset{i\rightarrow \infty }{\lim }\beta _{i}=0$ and $\alpha
_{i}\in \mathbb{N}_{0}$, $i\geq 1$ for $\alpha \in \mathcal{I}$ we see that
the set $A_{R}$ only has finitely many elements. On the other hand we obtain 
\begin{align}
\lVert G\rVert _{L^{2}(\Omega )}^{2}& =\sum_{\alpha }\alpha !c_{\alpha
}^{2}=\sum_{\alpha \notin A_{R}}\alpha !c_{\alpha }^{2}+\sum_{\alpha \in
A_{R}}\alpha !c_{\alpha }^{2}  \label{eq:norm} \\
& =\frac{R}{R}\sum_{\alpha \notin A_{R}}\alpha !c_{\alpha }^{2}+\sum_{\alpha
\in A_{R}}\alpha !c_{\alpha }^{2}  \notag \\
& \leq \frac{1}{R}\sum_{\alpha \notin A_{R}}\alpha !c_{\alpha }^{2}\sum_{k}%
\frac{1}{\beta _{k}^{2}}\alpha _{k}+\sum_{\alpha \in A_{R}}\alpha !c_{\alpha
}^{2}  \notag \\
& \leq \frac{1}{R}\sum_{\alpha }\alpha !c_{\alpha }^{2}\sum_{k}\frac{1}{%
\beta _{k}^{2}}\alpha _{k}+\sum_{\alpha \in A_{R}}\alpha !c_{\alpha }^{2} 
\notag \\
& =\frac{1}{R}\lVert C^{-1}D_{\cdot ,\cdot }G\rVert _{L^{2}(\Omega ;H\otimes
L^{2}(\nu ))}^{2}+\sum_{\alpha \in A_{R}}\alpha !c_{\alpha }^{2}.  \notag
\end{align}%
Let $\epsilon >0$. Since $A_{R}$ is finite we can find $C_{\alpha }^{j}$, $%
\alpha \in A_{R}$, $1\leq j\leq n(R,\epsilon )$, such that for all $G\in 
\mathcal{G}$ 
\begin{equation*}
\inf_{j}\left\{ \sum_{\alpha \in A_{R}}\alpha !\lvert c_{\alpha }-c_{\alpha
}^{j}\rvert ^{2}\right\} <\frac{\epsilon }{2}.
\end{equation*}%
Define 
\begin{equation*}
G^{j}:=\sum_{\alpha \in A_{R}}c_{\alpha }^{j}H_{\alpha },
\end{equation*}%
and replace $G$ in (\ref{eq:norm}) by $G-G^{j}$. Then we see that 
\begin{equation*}
\inf_{j}\lVert G-G^{j}\rVert _{L^{2}(\Omega )}^{2}\leq \frac{1}{R}\lVert
C^{-1}D_{\cdot ,\cdot }G\rVert _{L^{2}(\Omega ;H\otimes L^{2}(\nu ))}^{2}+%
\frac{\epsilon }{2}\leq \epsilon ,
\end{equation*}%
$j=1,\ldots ,n(R,\epsilon )$ for $R\geq 2\frac{c%
{{}^2}%
}{\epsilon }$. So the $L^{2}(\Omega )$-balls with center $G^{j}$, $%
j=1,\ldots ,n(R,\epsilon )$ and radius $\epsilon $ cover $\mathcal{G}$.
\end{proof}

\bigskip

\subsection{Some Regularity Results}

\subsubsection{Kolmogorov Type Equations Associated with L\'{e}vy Processes}

In this subsection we want to prove some regularity results for Kolmogorov
type equations associated with certain L\'{e}vy processes. The latter
results will be used to recast the drift term $\int_{0}^{t}b(s,X_{s})ds$ in
the SDE (\ref{eq:main}) in terms of a more regular expression which enables
us to compute certain estimates with respect to the Malliavin derivative of
approximating solutions to $X_{\cdot }$ (see Section 4).

\bigskip

We need the following lemma:

\begin{lemma}
\label{Lemma 4} Let $L_{t},0\leq t\leq T$ be a L\'{e}vy process and let $%
\phi :[0,T]\times \mathbb{R}^{d}\rightarrow \mathbb{R}$ be a bounded
measurable function such that $\phi (\cdot ,x)$ is continuous for all $x$
and such that $\phi (t,\cdot )\in Dom(\mathcal{L})$ for all $t$, where $%
\mathcal{L}$ is the generator of $L_{t},0\leq t\leq T$. Consider 
\begin{equation}
u(t,x)=\int_{0}^{t}E[\phi (s,x+L_{t-s})]ds.  \label{(6)}
\end{equation}%
Then $u$ solves 
\begin{equation*}
\frac{\partial u}{\partial t}=\mathcal{L}_{t}u+\phi
\end{equation*}%
with $u(0,x)=0$.
\end{lemma}

\begin{proof}
Denote by $\{P_{t}\}_{t\geq 0}$ the strongly continuous semigroup on $%
C_{\infty }(\mathbb{R}^{d})$ (space of continuous functions vanishing at
infinity) associated with our L\'{e}vy process $L_{t},0\leq t\leq T$, that
is 
\begin{equation*}
P_{t}(f)(x):=E[f(x+L_{t})]
\end{equation*}%
for $f\in C_{\infty }(\mathbb{R}^{d})$. See e.g. \cite{Applebaum}. If $f\in
Dom(\mathcal{L})$ we know that $P_{t}f$ solves the heat equation 
\begin{equation*}
\frac{d}{dt}P_{t}f=\mathcal{L}P_{t}f.
\end{equation*}%
Then if follows from the linearity of the operator $\mathcal{L}$ that 
\begin{equation*}
u(t,x)=\int_{0}^{t}E[\phi (s,x+L_{t-s})]ds,
\end{equation*}%
solves the Kolmogorov equation 
\begin{equation*}
\frac{\partial u}{\partial t}(t,x)=\mathcal{L}u(t,x)+\phi (t,x),
\end{equation*}%
with $u(0,x)=0$ for all $x$.
\end{proof}

\begin{remark}
We mention that the Schwartz test function space $S(\mathbb{R}^{d})$ is
contained in $Dom(\mathcal{L})$. See \cite{Applebaum}.
\end{remark}

\bigskip

In what follows we want to consider L\'{e}vy processes $L_{t},0\leq t\leq T$
given by truncated $\alpha $-stable processes of index $\alpha \in (0,2)$,
that is L\'{e}vy processes, whose characteristic exponent is given by 
\begin{equation}
\Psi (u)=\int_{\mathbb{R}^{d}}(1-\cos (u\cdot y))\nu (dy),
\end{equation}%
with L\'{e}vy measure%
\begin{equation*}
\nu (dy)=\boldsymbol{1}_{\{\left\vert y\right\vert \leq 1\}}\frac{1}{%
\left\vert y\right\vert ^{d+\alpha }}dy.
\end{equation*}%
See e.g. \cite{KS} for further properties of this process.

\QTP{bTë?}
Note that the infinitesimal generator $\mathcal{L}$ of the process $L$ is
given by%
\begin{equation}
\mathcal{L}f(x)=\int_{\mathbb{R}^{d}}(f(x+y)-f(x))-\boldsymbol{1}%
_{\{\left\vert y\right\vert \leq 1\}}y\cdot Df(x)\nu (dy)  \label{Generator}
\end{equation}%
for $f\in C_{c}^{\infty }(\mathbb{R}^{d})$ (space of infinitely
differentiable functions with compact support). See e.g. \cite{Applebaum} .

\QTP{bTë?}
\bigskip

\QTP{bTë?}
We need the following auxiliary result:

\begin{lemma}
\label{Lemma 6} Let $f\in W_{1}^{r_{1},\ldots ,r_{d}}(\mathbb{R}^{d})$, $%
d\geq 2$ and let $r_{1},\ldots ,r_{d}\in \mathbb{N}$ such that 
\begin{equation*}
\sum_{i=1}^{d}\frac{1}{r_{i}}=1.
\end{equation*}%
Then 
\begin{equation*}
\lVert \hat{f}\rVert _{L^{1}(\mathbb{R}^{d})}\leq C\sum_{j=1}^{d}\lVert 
\frac{\partial ^{r_{j}}}{\partial x_{j}^{r_{j}}}f\rVert _{L^{1}(\mathbb{R}%
^{d})},
\end{equation*}%
where $\hat{f}$ denotes the Fourier-transform of $f$.
\end{lemma}

\begin{proof}
See Remark 1 in \cite{Kolyada}.
\end{proof}

\begin{theorem}
\label{Theorem 5}\bigskip Let $L_{t},0\leq t\leq T$ be a $d$-dimensional
truncated $\alpha -$stable process for $\alpha \in (1,2)$ and $d\geq 2$.
Suppose that $\phi \in C([0,T],C_{b}^{\beta }(\mathbb{R}^{d}))$ for $\beta
\in (0,1)$ satisfies that $\alpha +\beta >2$. Then there exists a $u\in
C([0,T],C_{b}^{2}(\mathbb{R}^{d}))\cap C^{1}([0,T],C_{b}(\mathbb{R}^{d}))$
such that%
\begin{equation}
\frac{\partial u}{\partial t}=\mathcal{L}u+\phi ,  \label{Keq}
\end{equation}%
with $\mathcal{L}$ defined as in (\ref{Generator}) and such that%
\begin{equation}
\left\Vert Du\right\Vert _{C_{b}^{\beta }}\leq C(T)\left\Vert \phi
\right\Vert _{C_{b}^{\beta }},  \label{Est1}
\end{equation}%
where 
\begin{equation*}
C(T)\longrightarrow 0\text{ for }T\searrow 0,
\end{equation*}%
as well as 
\begin{equation}
\left\Vert D^{2}u\right\Vert _{\infty }\leq M\left\Vert \phi \right\Vert
_{C_{b}^{\beta }}  \label{Est2}
\end{equation}%
for a constant $M$.
\end{theorem}

\begin{proof}
We subdivide the proof into two parts:

$\mathbf{(A)}$ We first want to show that $u$ defined by (\ref{(6)}) in
Lemma \ref{Lemma 4} admits the estimates (\ref{Est1}) and (\ref{Est2}):

We recall that $L_{t},0\leq t\leq T$ has the characteristic exponent%
\begin{equation*}
\Psi (u)=\int_{\left\vert y\right\vert \leq 1}\frac{1-\cos (u\cdot y)}{%
\left\vert y\right\vert ^{d+\alpha }}dy.
\end{equation*}%
So we get%
\begin{eqnarray}
t\Psi (t^{-\frac{1}{\alpha }}u) &=&\int_{\left\vert y\right\vert \leq 1}%
\frac{1-\cos (t^{-\frac{1}{\alpha }}u\cdot y)}{\left\vert y\right\vert
^{d+\alpha }}dy  \notag \\
&&\overset{y=t^{\frac{1}{\alpha }}r}{=}\int_{\left\vert r\right\vert \leq
t^{-\frac{1}{\alpha }}}t\cdot t^{\frac{d}{\alpha }}\cdot t^{-\frac{d+\alpha 
}{\alpha }}\frac{1-\cos (u\cdot r)}{\left\vert r\right\vert ^{d+\alpha }}dr 
\notag \\
&=&\int_{\left\vert r\right\vert \leq t^{-\frac{1}{\alpha }}}\frac{1-\cos
(u\cdot r)}{\left\vert r\right\vert ^{d+\alpha }}dr  \notag \\
&\geq &\int_{\left\vert r\right\vert \leq T^{-\frac{1}{\alpha }}}\frac{%
1-\cos (u\cdot r)}{\left\vert r\right\vert ^{d+\alpha }}dr  \notag \\
&:&=\widetilde{\Psi }(u).  \label{(7)}
\end{eqnarray}

Observe that%
\begin{equation}
\widetilde{\Psi }(u)\sim \left\vert u\right\vert ^{\alpha }  \label{(8)}
\end{equation}%
nearby infinity. Because of (\ref{(7)}) and (\ref{(8)}) we can apply the
Fourier inversion formula and obtain for the probability density function $%
p_{t}(x),$ the representation%
\begin{equation*}
p_{t}(x)=\frac{1}{(2\pi )^{d}}\int_{\mathbb{R}^{d}}e^{-ixu}e^{-t\Psi (u)}du.
\end{equation*}%
Hence%
\begin{eqnarray*}
p_{t}(t^{\frac{1}{\alpha }}x) &=&\frac{1}{(2\pi )^{d}}\int_{\mathbb{R}%
^{d}}e^{-it^{\frac{1}{\alpha }}xz}e^{-t\Psi (z)}dz \\
&&\overset{z=t^{-\frac{1}{\alpha }}u}{=}\frac{t^{-\frac{d}{\alpha }}}{(2\pi
)^{d}}\int_{\mathbb{R}^{d}}e^{-ixu}e^{-t\Psi (t^{-\frac{1}{\alpha }}u)}du.
\end{eqnarray*}%
Because of (\ref{(7)}) and (\ref{(8)}) we are allowed to differentiate $%
p_{t}(\cdot )$ and get%
\begin{eqnarray}
\frac{\partial }{\partial x_{i}}(p_{t}(t^{\frac{1}{\alpha }}x)) &=&t^{\frac{1%
}{\alpha }}(\frac{\partial }{\partial x_{i}}p_{t})(t^{\frac{1}{\alpha }}x) 
\notag \\
&=&\frac{t^{-\frac{d}{\alpha }}}{(2\pi )^{d}}\int_{\mathbb{R}%
^{d}}e^{-ixu}(-i)(u_{i})e^{-t\Psi (t^{-\frac{1}{\alpha }}u)}du.  \label{(9)}
\end{eqnarray}%
On the other hand we know that%
\begin{eqnarray*}
E[\phi (s,x+L_{t})] &=&\int_{\mathbb{R}^{d}}\phi (s,x+u)p_{t}(u)du \\
&=&\int_{\mathbb{R}^{d}}\phi (s,u)p_{t}(u-x)du.
\end{eqnarray*}%
So we get%
\begin{eqnarray}
\frac{\partial }{\partial x_{i}}E[\phi (s,x+L_{t})] &=&-\int_{\mathbb{R}%
^{d}}\phi (s,u)\frac{\partial }{\partial x_{i}}p_{t}(u-x)du  \notag \\
&=&-\int_{\mathbb{R}^{d}}\phi (s,u+x)(\frac{\partial }{\partial x_{i}}%
p_{t})(u)du  \notag \\
&=&-\int_{\mathbb{R}^{d}}\phi (s,t^{\frac{1}{\alpha }}u+x)t^{\frac{d}{\alpha 
}}(\frac{\partial }{\partial x_{i}}p_{t})(t^{\frac{1}{\alpha }}u)du  \notag
\\
&=&-\int_{\mathbb{R}^{d}}\phi (s,t^{\frac{1}{\alpha }}u+x)t^{\frac{d-1}{%
\alpha }}\frac{\partial }{\partial x_{i}}(p_{t}(t^{\frac{1}{\alpha }}u))du.
\label{(10)}
\end{eqnarray}%
In order to give an estimate of the $L^{1}-$norm of $\frac{\partial }{%
\partial x_{i}}(p_{t}(t^{\frac{1}{\alpha }}\cdot )),$ $i=1,...,d$ in (\ref%
{(9)}) we want to apply Lemma \ref{Lemma 6}. Without loss of generality let
us consider the case $d=2$. Then, using Lemma \ref{Lemma 6} we find%
\begin{equation*}
\left\Vert \frac{\partial }{\partial x_{i}}(p_{t}(t^{\frac{1}{\alpha }}\cdot
))\right\Vert _{L^{1}(\mathbb{R}^{d})}\leq Ct^{-\frac{d}{\alpha }%
}\sum_{j=1}^{d}\left\Vert \frac{\partial ^{2}}{\partial u_{j}^{2}}\eta
_{i}\right\Vert _{L^{1}(\mathbb{R}^{d})},
\end{equation*}%
where%
\begin{equation*}
\eta _{i}(u)=u_{i}e^{-t\Psi (t^{-\frac{1}{\alpha }}u)}.
\end{equation*}%
Let $i\neq j$. Then%
\begin{eqnarray*}
\frac{\partial ^{2}}{\partial u_{j}^{2}}\eta _{i}(u) &=&u_{i}t^{1-\frac{2}{%
\alpha }}(\frac{\partial ^{2}}{\partial u_{j}^{2}}\Psi )(t^{-\frac{1}{\alpha 
}}u)e^{-t\Psi (t^{-\frac{1}{\alpha }}u)} \\
&&+u_{i}(t^{1-\frac{1}{\alpha }}(\frac{\partial }{\partial u_{j}}\Psi )(t^{-%
\frac{1}{\alpha }}u))^{2}e^{-t\Psi (t^{-\frac{1}{\alpha }}u)}.
\end{eqnarray*}%
We know that%
\begin{eqnarray*}
(\frac{\partial }{\partial u_{j}}\Psi )(u) &=&\int_{\left\vert y\right\vert
\leq 1}y_{j}\frac{\sin (u\cdot y)}{\left\vert y\right\vert ^{d+\alpha }}dy \\
&&\overset{y=\frac{1}{\left\vert u\right\vert }r}{=}\int_{\left\vert
r\right\vert \leq \left\vert u\right\vert }\frac{1}{\left\vert u\right\vert
^{d+1}}\left\vert u\right\vert ^{d+\alpha }r_{j}\frac{\sin (\frac{u}{%
\left\vert u\right\vert }\cdot r)}{\left\vert r\right\vert ^{d+\alpha }}dr \\
&=&\left\vert u\right\vert ^{\alpha -1}\int_{\left\vert r\right\vert \leq
\left\vert u\right\vert }r_{j}\frac{\sin (\frac{u}{\left\vert u\right\vert }%
\cdot r)}{\left\vert r\right\vert ^{d+\alpha }}dr.
\end{eqnarray*}%
On the other hand we see that%
\begin{eqnarray*}
\left\vert \int_{\left\vert r\right\vert \leq \left\vert u\right\vert }r_{j}%
\frac{\sin (\frac{u}{\left\vert u\right\vert }\cdot r)}{\left\vert
r\right\vert ^{d+\alpha }}dr\right\vert &\leq &\int_{\left\vert r\right\vert
\leq \left\vert u\right\vert }\left\vert r_{j}\right\vert \frac{\left\vert
\sin (\frac{u}{\left\vert u\right\vert }\cdot r)\right\vert }{\left\vert
r\right\vert ^{d+\alpha }}dr \\
&\leq &\int_{\left\vert r\right\vert \leq \left\vert u\right\vert
}\left\vert r\right\vert \frac{\left\vert \sin (\frac{u}{\left\vert
u\right\vert }\cdot r)\right\vert }{\left\vert r\right\vert ^{d+\alpha }}dr
\\
&\leq &\int_{\left\vert r\right\vert \leq \left\vert u\right\vert }\frac{%
\left\vert \sin (\frac{u}{\left\vert u\right\vert }\cdot r)\right\vert }{%
\left\vert r\right\vert ^{d+\alpha -1}}dr \\
&\leq &M<\infty
\end{eqnarray*}%
for all $u$. So%
\begin{eqnarray}
t^{1-\frac{1}{\alpha }}\left\vert (\frac{\partial }{\partial u_{j}}\Psi
)(t^{-\frac{1}{\alpha }}u)\right\vert &\leq &Mt^{1-\frac{1}{\alpha }}(t^{-%
\frac{1}{\alpha }})^{\alpha -1}\left\vert u\right\vert ^{\alpha -1}  \notag
\\
&=&M\left\vert u\right\vert ^{\alpha -1}  \label{(11)}
\end{eqnarray}%
for all $u$. Further we have that%
\begin{eqnarray}
\left\vert \frac{\partial ^{2}}{\partial u_{j}^{2}}\Psi (u)\right\vert &\leq
&\int_{\left\vert y\right\vert \leq 1}y_{j}^{2}\frac{\left\vert \cos (u\cdot
y)\right\vert }{\left\vert y\right\vert ^{d+\alpha }}dy  \notag \\
&\leq &\int_{\left\vert y\right\vert \leq 1}\frac{1}{\left\vert y\right\vert
^{d+\alpha -2}}dy=C<\infty  \label{(13)}
\end{eqnarray}%
for all $u$. Using (\ref{(7)}), (\ref{(8)}) and (\ref{(13)}), it follows that%
\begin{eqnarray*}
\left\Vert \frac{\partial ^{2}}{\partial u_{j}^{2}}\eta _{i}\right\Vert
_{L^{1}(\mathbb{R}^{d})} &\leq &t^{1-\frac{2}{\alpha }}C\left\Vert u_{i}e^{-%
\widetilde{\Psi }(u)}\right\Vert _{L^{1}(\mathbb{R}^{d})} \\
&&+M\left\Vert u_{i}\left\vert u\right\vert ^{2(\alpha -1)}e^{-\widetilde{%
\Psi }(u)}\right\Vert _{L^{1}(\mathbb{R}^{d})}.
\end{eqnarray*}%
Similarly, we can treat the case $i=j$ and find%
\begin{eqnarray*}
\left\Vert \frac{\partial }{\partial x_{i}}(p_{t}(t^{\frac{1}{\alpha }}\cdot
))\right\Vert _{L^{1}(\mathbb{R}^{d})} &\leq &C_{1}t^{1+\frac{1-d}{\alpha }%
}+C_{2}t^{-\frac{d}{\alpha }} \\
&\leq &Ct^{-\frac{d}{\alpha }}
\end{eqnarray*}%
for constants $C_{1},C_{2}$ and $C<\infty $. This estimate and (\ref{(10)})
then give%
\begin{eqnarray*}
\left\vert \frac{\partial }{\partial x_{i}}E[\phi (s,x+L_{t})]\right\vert
&\leq &C\left\Vert \phi \right\Vert _{\infty }t^{\frac{d-1}{\alpha }}t^{-%
\frac{d}{\alpha }} \\
&=&C\left\Vert \phi \right\Vert _{\infty }t^{-\frac{1}{\alpha }}
\end{eqnarray*}%
for all $x$ and $s$.

So%
\begin{eqnarray*}
\left\vert \frac{\partial }{\partial x_{i}}u(t,x)\right\vert &\leq
&C\left\Vert \phi \right\Vert _{\infty }\int_{0}^{t}(t-s)^{-\frac{1}{\alpha }%
}ds \\
&\leq &C(T)\left\Vert \phi \right\Vert _{\infty }
\end{eqnarray*}%
for all $x,s$ with%
\begin{equation*}
C(T)\longrightarrow 0\text{ for }T\searrow 0\text{.}
\end{equation*}
Using the same arguments just as above we also get%
\begin{equation*}
\frac{\left\vert \frac{\partial }{\partial x_{i}}u(t,x)-\frac{\partial }{%
\partial x_{i}}u(t,y)\right\vert }{\left\vert x-y\right\vert ^{\beta }}\leq
C(T)\left\Vert \phi \right\Vert _{C_{b}^{\beta }}
\end{equation*}%
for all $x\neq y$ and $s$.

Let us now derive an estimate with respect to $D^{2}u$. \ We observe that%
\begin{eqnarray}
\frac{\partial ^{2}}{\partial x_{i}^{2}}E[\phi (s,x+L_{t})] &=&\int_{\mathbb{%
R}^{d}}\phi (s,u)\frac{\partial ^{2}}{\partial x_{i}^{2}}p_{t}(u-x)du  \notag
\\
&=&\int_{\mathbb{R}^{d}}\phi (s,u+x)(\frac{\partial ^{2}}{\partial x_{i}^{2}}%
p_{t})(u)du  \notag \\
&=&\int_{\mathbb{R}^{d}}\phi (s,t^{\frac{1}{\alpha }}u+x)t^{\frac{d}{\alpha }%
}(\frac{\partial ^{2}}{\partial x_{i}^{2}}p_{t})(t^{\frac{1}{\alpha }}u)du 
\notag \\
&=&\int_{\mathbb{R}^{d}}\phi (s,t^{\frac{1}{\alpha }}u+x)t^{\frac{d-2}{%
\alpha }}\frac{\partial ^{2}}{\partial x_{i}^{2}}(p_{t}(t^{\frac{1}{\alpha }%
}u))du.  \label{(14)}
\end{eqnarray}%
On the other hand it follows from (\ref{(9)}) 
\begin{equation}
\frac{\partial ^{2}}{\partial x_{i}^{2}}(p_{t}(t^{\frac{1}{\alpha }}x))=-%
\frac{t^{-\frac{d}{\alpha }}}{(2\pi )^{d}}\int_{\mathbb{R}%
^{d}}e^{-ixu}u_{i}^{2}e^{-t\Psi (t^{-\frac{1}{\alpha }}u)}du.  \label{(15)}
\end{equation}%
Then it follows from Lemma 6, (\ref{(14)}) and (\ref{(15)}) by using the
same arguments as above that%
\begin{eqnarray}
\left\vert \frac{\partial ^{2}}{\partial x_{i}^{2}}E[\phi
(s,x+L_{t})]\right\vert &\leq &C\left\Vert \phi \right\Vert _{\infty }t^{%
\frac{d-2}{\alpha }}t^{-\frac{d}{\alpha }}  \notag \\
&=&C\left\Vert \phi \right\Vert _{\infty }t^{-\frac{2}{\alpha }}  \label{P1}
\end{eqnarray}%
for all $x,s$. The case of mixed partial derivatives can be treated
similarly and we obtain%
\begin{equation}
\left\Vert D^{2}P_{t}\phi \right\Vert _{\infty }\leq C\left\Vert \phi
\right\Vert _{\infty }t^{-\frac{2}{\alpha }}  \label{(16)}
\end{equation}%
for all $\phi \in C_{b}(\mathbb{R}^{d})$ as well as%
\begin{equation}
\left\Vert D^{2}P_{t}\phi \right\Vert _{\infty }\leq C\left\Vert D\phi
\right\Vert _{\infty }t^{-\frac{1}{\alpha }}  \label{(17)}
\end{equation}%
for all $\phi \in C_{b}^{1}(\mathbb{R}^{d})$, where we used the semi-group
notation%
\begin{equation*}
(P_{t}\phi )(x)=E[\phi (x+L_{t})].
\end{equation*}%
Further, from interpolation theory (see e.g. \cite{Lunardi}), it is known
that%
\begin{equation*}
\left( C_{b}(\mathbb{R}^{d}),C_{b}^{1}(\mathbb{R}^{d})\right) _{\beta
,\infty }=C_{b}^{\beta }(\mathbb{R}^{d}).
\end{equation*}%
So using (\ref{(16)}) and (\ref{(17)}) in connection with Theorem 1.1.6 in 
\cite{Lunardi} one finds that%
\begin{eqnarray}
\left\Vert D^{2}P_{t}\phi \right\Vert _{\infty } &\leq &C\frac{1}{t^{\frac{%
2(1-\beta )}{\alpha }}}\frac{1}{t^{\frac{\beta }{\alpha }}}\left\Vert \phi
\right\Vert _{C_{b}^{\beta }(\mathbb{R}^{d})}  \notag \\
&=&C\frac{1}{t^{\frac{(2-\beta )}{\alpha }}}\left\Vert \phi \right\Vert
_{C_{b}^{\beta }(\mathbb{R}^{d})}  \label{P2}
\end{eqnarray}%
for all $\phi \in C_{b}^{\beta }(\mathbb{R}^{d})$, where $C$ is a constant
depending on $\beta $. Since by assumption $\alpha +\beta >2$ we get that%
\begin{equation}
\left\Vert D^{2}u\right\Vert _{\infty }\leq C\left\Vert \phi \right\Vert
_{C_{b}^{\beta }}  \label{D2}
\end{equation}%
for all $\phi \in C([0,T],C_{b}^{\beta }(\mathbb{R}^{d}))$.

$\mathbf{(B)}$ We aim at showing that $u$ defined by (\ref{(6)}) actually
solves the equation (\ref{Keq}) fo such $\phi $ as stated in the theorem:

We observe that%
\begin{eqnarray}
\left\vert f(y+x)-f(x)-y\cdot Df(x)\right\vert &=&\left\vert
\int_{0}^{1}(Df(x+\theta y)-Df(x))\cdot yd\theta \right\vert  \notag \\
&\leq &\left\Vert D^{2}f\right\Vert _{\infty }\left\vert y\right\vert ^{2}
\label{IneqLevy}
\end{eqnarray}%
for all $x$ and $y$ with $\left\vert y\right\vert \leq 1$. Using this
inequality we see that $\mathcal{L}f\in C_{b}(\mathbb{R}^{d})$, if $f\in
C_{b}^{2}(\mathbb{R}^{d}).$

So it follows from part $\mathbf{(A)}$ and the inequality (\ref{IneqLevy})
that $\mathcal{L}u$ is well-defined. Further, one has that $C_{\infty }^{2}(%
\mathbb{R}^{d})\subset Dom(\mathcal{L}),$ where $C_{\infty }^{2}(\mathbb{R}%
^{d}):=C_{b}^{2}(\mathbb{R}^{d})\cap C_{0}(\mathbb{R}^{d}).$ See e.g. \cite%
{Applebaum}. Hence, if $\phi \in C([0,T],C_{\infty }^{2}(\mathbb{R}^{d}))$
then we know by Lemma \ref{Lemma 4} that $u$ given by (\ref{(6)}) satisfies (%
\ref{Keq}).

Let us now assume that $\phi \in C([0,T],C_{b}^{2}(\mathbb{R}^{d})).$ Choose
a $\varphi \in C_{c}^{\infty }(\mathbb{R}^{d})$ with $\varphi (0)=1$ and set 
$\phi _{n}(t,x)=\varphi (x/n)\phi (t,x)$ for $0\leq t\leq T,x\in \mathbb{R}%
^{d}$ and $n\geq 1$. We know from the proofs in $\mathbf{(A)}$ that $u\in
C([0,T],C_{b}^{2}(\mathbb{R}^{d})).$ So using this one verifies that $\phi
_{n},u_{n}\in C([0,T],C([0,T],C_{b}^{2}(\mathbb{R}^{d}))$ for all $n,$ where 
$u_{n}(t,x):=$ $\varphi (x/n)u(t,x).$ Hence $u_{n}$ satisfies (\ref{Keq})
for all $n\geq 1$. Further, one sees that $\phi _{n}\longrightarrow \phi $, $%
D\phi _{n}$ $\longrightarrow D\phi $ pointwise and that for a constant $C$: $%
\left\Vert \phi _{n}\right\Vert _{2}\leq C$ for all $n\geq 1.$ On the other
hand by (\ref{(10)}) we obtain pointwise convergence of $Du_{n}$ to $Du$. So
using dominated convergence in connection with the estimates with (\ref%
{IneqLevy}) and (\ref{D2}) we find that $\mathcal{L}u_{n}$ converges
pointwise to $\mathcal{L}u$ for $n\longrightarrow \infty $. From this we can
see that $u$ for $\phi \in C([0,T],C_{b}^{2}(\mathbb{R}^{d}))$ solves (\ref%
{Keq}).

Finally, consider the case $\phi \in C([0,T],C_{b}^{\beta }(\mathbb{R}^{d}))$%
. Here we apply an approximation argument which can be found e.g. in the
book \cite{Krylov}: Let $\varphi _{n}(x)=n^{d}\phi (xn)$, where $\varphi \in
C_{c}^{\infty }(\mathbb{R}^{d})$ such that $\varphi (x)\in \lbrack 0,1]$ for
all $x$ and $\int_{\mathbb{R}^{d}}$ $\varphi (x)dx=1$. Define%
\begin{equation*}
\phi _{n}(t,x)=(\phi (t,\cdot )\ast \varphi _{n})(x)
\end{equation*}%
($f\ast g$ convolution of functions $f$ and $g$). One obtains that $\phi
_{n}\in C([0,T],C_{b}^{\infty }(\mathbb{R}^{d}))$ and%
\begin{equation}
\left\Vert \phi _{n}\right\Vert _{C_{b}^{\beta }}\leq \left\Vert \phi
\right\Vert _{C_{b}^{\beta }}  \label{KrylovApprox}
\end{equation}%
for all $n$. Further, we have that 
\begin{equation}
\phi _{n_{k}(t)}(t,\cdot )\longrightarrow \phi (t,\cdot )\text{ in }%
C^{\delta }(K)  \label{KrylovConv}
\end{equation}%
for all $t$, any compact set $K\subset \mathbb{R}^{d}$ and $0<\delta <\beta $
for a subsequence $n_{k}(t),k\geq 1$ depending on $t$ and $K$. See e.g. \cite%
{Krylov}. Let $K_{m},m\geq 1$ be an increasing sequence of compact sets such
that $\cup _{m\geq 1}K_{m}=\mathbb{R}^{d}$. Then for each $K_{m}$ there
exists a subsequence $n_{k}(m,t),k\geq 1$ such that (\ref{KrylovConv}) holds
. Then by choosing a diagonal sequence $n_{k}^{\ast }(t),k\geq 1$ with
respect to $n_{k}(1,t),k\geq 1,n_{k}(2,t),k\geq 1,...$ we conclude from (\ref%
{(10)}) in connection with dominated convergence that%
\begin{equation*}
(P_{r}\phi _{n_{k}^{\ast }(t)}(t,\cdot ))(x)\longrightarrow (P_{r}\phi
(t,\cdot ))(x),
\end{equation*}%
\begin{equation*}
(DP_{r}\phi _{n_{k}^{\ast }(t)}(t,\cdot ))(x)\longrightarrow (DP_{r}\phi
(t,\cdot ))(x)
\end{equation*}%
pointwise in $x$ for all $r$,$t.$ So using (\ref{(16)}), (\ref{IneqLevy})
and dominated convergence we get%
\begin{equation*}
(\mathcal{L}P_{r}\phi _{n_{k}^{\ast }(t)}(t,\cdot ))\longrightarrow (%
\mathcal{L}P_{r}\phi (t,\cdot ))(x)
\end{equation*}%
pointwise in $x$ for all $r,t.$ On the other hand we can argue as above and
find that%
\begin{equation*}
\frac{\partial }{\partial r}P_{r}\phi _{n_{k}^{\ast }(t)}(t,\cdot )=\mathcal{%
L}P_{r}\phi _{n_{k}^{\ast }(t)}(t,\cdot ).
\end{equation*}%
By employing dominated convergence we obtain that%
\begin{equation*}
\frac{\partial }{\partial r}P_{r}\phi (t,\cdot )=\mathcal{L}P_{r}\phi
(t,\cdot ).
\end{equation*}%
Then, using the proof of Lemma \ref{Lemma 4} we see that $u$ satisfies (\ref%
{Keq}) for $\phi \in C([0,T],C_{b}^{\beta }(\mathbb{R}^{d}))$.

Finally, by applying (\ref{P1}), (\ref{P2}), (\ref{IneqLevy}) and%
\begin{equation}
u(t,x)=\int_{0}^{t}((\mathcal{L}u)(s,x)+\phi (s,x))ds,
\end{equation}
in connection with dominated convergence, we see that $u\in
C^{1}([0,T],C_{b}(\mathbb{R}^{d})).$
\end{proof}

\bigskip

\bigskip

\begin{theorem}
\label{Theorem 7}\bigskip Let $L_{t},0\leq t\leq T$ be a $d$-dimensional
truncated $\alpha -$stable process for $\alpha \in (1,2)$ and $d\geq 2$.
Require that $\varphi \in C([0,T],C_{b}^{\beta }(\mathbb{R}^{d}))$ for $%
\beta \in (0,1)$ with $\alpha +\beta >2$. Then there exists a $u\in
C([0,T],C_{b}^{2}(\mathbb{R}^{d}))\cap C^{1}([0,T],C_{b}(\mathbb{R}^{d}))$
satisfying the backward Kolmogorov equation%
\begin{eqnarray}
\frac{\partial u}{\partial t}+b\cdot \nabla u+\mathcal{L}u &=&-\varphi \text{
on }[0,T],  \notag \\
\left. u\right\vert _{t=T} &=&0\text{.}  \label{BackwardK}
\end{eqnarray}%
Moreover%
\begin{equation}
\left\Vert Du\right\Vert _{C_{b}^{\beta }}\leq C(T)\left\Vert \varphi
\right\Vert _{C_{b}^{\beta }},  \label{K1}
\end{equation}%
where 
\begin{equation*}
C(T)\longrightarrow 0\text{ for }T\searrow 0,
\end{equation*}%
as well as 
\begin{equation}
\left\Vert D^{2}u\right\Vert _{\infty }\leq M\left\Vert \varphi \right\Vert
_{C_{b}^{\beta }}  \label{K2}
\end{equation}%
for a constant $M$.
\end{theorem}

\begin{proof}
We want to use Picard iteration based on (\ref{Keq}) to construct a solution
to (\ref{BackwardK}) (compare Theorem 2.8 in \cite{Flandoli} in the case of
Brownian motion): Let $u^{0}=0$ and define for $n\geq 0$%
\begin{eqnarray}
\frac{\partial u^{n+1}}{\partial t}+\mathcal{L}u^{n+1} &=&-(b\cdot \nabla
u^{n})-\varphi \text{ on }[0,T],  \notag \\
\left. u^{n+1}\right\vert _{t=T} &=&0\text{.}  \label{Picard}
\end{eqnarray}%
Since $u$ in Theorem \ref{Theorem 5} belongs to $C([0,T],C_{b}^{2}(\mathbb{R}%
^{d}))\cap C^{1}([0,T],C_{b}(\mathbb{R}^{d})),$ we see from (\ref{Est1}) that%
\begin{eqnarray*}
\left\Vert b\cdot \nabla u+\varphi \right\Vert _{C_{b}^{\beta }} &\leq
&\left\Vert \varphi \right\Vert _{C_{b}^{\beta }}+2\left\Vert b\right\Vert
_{C_{b}^{\beta }}\left\Vert \nabla u\right\Vert _{C_{b}^{\beta }} \\
&\leq &\left\Vert \varphi \right\Vert _{C_{b}^{\beta }}+2\left\Vert
b\right\Vert _{C_{b}^{\beta }}C(T)\left\Vert \phi \right\Vert _{C_{b}^{\beta
}} \\
&<&\infty
\end{eqnarray*}%
for $\phi \in C([0,T],C_{b}^{\beta }(\mathbb{R}^{d}))$. So it follows from
Theorem \ref{Theorem 5} that we obtain in each iteration step a solution%
\begin{equation*}
u^{n+1}\in C([0,T],C_{b}^{2}(\mathbb{R}^{d}))\cap C^{1}([0,T],C_{b}(\mathbb{R%
}^{d})).
\end{equation*}%
Let us now choose a $T>0$ in Theorem \ref{Theorem 5} such that%
\begin{equation*}
2C(T)\left\Vert b\right\Vert _{C_{b}^{\beta }}\leq \frac{1}{2}.
\end{equation*}%
Then, using the estimates (\ref{P1}) and (\ref{P2}) in the proof of Theorem %
\ref{Theorem 5} we find for all $n\geq 0$ that%
\begin{eqnarray*}
&&\left\Vert \nabla u^{n+1}-\nabla u^{n}\right\Vert _{C_{b}^{\beta
}}=\left\Vert \nabla (u^{n+1}-u^{n})\right\Vert _{C_{b}^{\beta }} \\
&\leq &C(T)\left\Vert (b\cdot \nabla u^{n}+\varphi )-(b\cdot \nabla
u^{n-1}+\varphi )\right\Vert _{C_{b}^{\beta }} \\
&=&C(T)\left\Vert b\cdot (\nabla u^{n}-\nabla u^{n-1})\right\Vert
_{C_{b}^{\beta }} \\
&\leq &2C(T)\left\Vert b\right\Vert _{C_{b}^{\beta }}\left\Vert \nabla
u^{n}-\nabla u^{n-1}\right\Vert _{C_{b}^{\beta }} \\
&\leq &\frac{1}{2}\left\Vert \nabla u^{n}-\nabla u^{n-1}\right\Vert
_{C_{b}^{\beta }} \\
&&... \\
&\leq &(\frac{1}{2})^{n}\left\Vert \nabla u\right\Vert _{C_{b}^{\beta }} \\
&\leq &(\frac{1}{2})^{n}C(T)\left\Vert \varphi \right\Vert _{C_{b}^{\beta }},
\end{eqnarray*}%
\begin{eqnarray*}
&&\left\Vert u^{n+1}-u^{n}\right\Vert _{C_{b}^{\beta }} \\
&\leq &K\left\Vert b\cdot (\nabla u^{n}-\nabla u^{n-1})\right\Vert
_{C_{b}^{\beta }} \\
&\leq &2K\left\Vert b\right\Vert _{C_{b}^{\beta }}\left\Vert \nabla
u^{n}-\nabla u^{n-1}\right\Vert _{C_{b}^{\beta }} \\
&\leq &2K\left\Vert b\right\Vert _{C_{b}^{\beta }}(\frac{1}{2}%
)^{n-1}C(T)\left\Vert \varphi \right\Vert _{C_{b}^{\beta }} \\
&\leq &K\left\Vert \varphi \right\Vert _{C_{b}^{\beta }}(\frac{1}{2})^{n}
\end{eqnarray*}%
as well as 
\begin{eqnarray*}
\left\Vert D^{2}u^{n+1}-D^{2}u^{n}\right\Vert _{\infty } &=&\left\Vert
D^{2}(u^{n+1}-u^{n})\right\Vert _{\infty } \\
&\leq &C(T)\left\Vert (b\cdot \nabla u^{n})+\varphi -((b\cdot \nabla
u^{n-1})+\varphi )\right\Vert _{C_{b}^{\beta }} \\
&=&C(T)\left\Vert b\cdot (\nabla u^{n}-\nabla u^{n-1})\right\Vert
_{C_{b}^{\beta }} \\
&\leq &2C(T)\left\Vert b\right\Vert _{C_{b}^{\beta }}\left\Vert \nabla
u^{n}-\nabla u^{n-1}\right\Vert _{C_{b}^{\beta }} \\
&\leq &\frac{1}{2}\left\Vert \nabla u^{n}-\nabla u^{n-1}\right\Vert
_{C_{b}^{\beta }} \\
&\leq &(\frac{1}{2})^{n}C(T)\left\Vert \varphi \right\Vert _{C_{b}^{\beta }}.
\end{eqnarray*}%
So we get that%
\begin{eqnarray*}
&&\left\Vert \nabla u^{n+p}-\nabla u^{n}\right\Vert _{C_{b}^{\beta }} \\
&\leq &\sum_{j=0}^{p-1}\left\Vert \nabla u^{n+p-j}-\nabla
u^{n+p-j-1}\right\Vert _{C_{b}^{\beta }} \\
&\leq &C(T)\left\Vert \varphi \right\Vert _{C_{b}^{\beta }}\sum_{j=0}^{p-1}(%
\frac{1}{2})^{n+p-j-1} \\
&=&C(T)\left\Vert \varphi \right\Vert _{C_{b}^{\beta }}(\frac{1}{2}%
)^{n}\sum_{j=0}^{p-1}(\frac{1}{2})^{j}
\end{eqnarray*}%
and similarly%
\begin{equation*}
\left\Vert u^{n+p}-u^{n}\right\Vert _{C_{b}^{\beta }}\leq K\left\Vert
\varphi \right\Vert _{C_{b}^{\beta }}(\frac{1}{2})^{n}\sum_{j=0}^{p-1}(\frac{%
1}{2})^{j}
\end{equation*}%
and 
\begin{equation*}
\left\Vert D^{2}u^{n+p}-D^{2}u^{n}\right\Vert _{\infty }\leq C(T)\left\Vert
\varphi \right\Vert _{C_{b}^{\beta }}(\frac{1}{2})^{n}\sum_{j=0}^{p-1}(\frac{%
1}{2})^{j}
\end{equation*}%
for all $n,p\geq 0$.

So we see that there exists a function $u$ such that 
\begin{equation*}
u^{n}\longrightarrow u,\text{ }\nabla u^{n}\longrightarrow \nabla u,\text{ }%
D^{2}u^{n}\longrightarrow D^{2}u\text{ }
\end{equation*}%
uniformly in $t,x$.

Using (\ref{IneqLevy}) and dominated convergence, we also observe that%
\begin{equation*}
\mathcal{L}u_{n}\longrightarrow \mathcal{L}u
\end{equation*}%
uniformly in $t,x$. Passing to the limit we obtain the equation 
\begin{equation}
u(t,x)=\int_{t}^{T}((\mathcal{L}u)(s,x)+b(s,x)(\nabla u)(s,x)+\varphi
(s,x))ds,n\geq 0  \label{IntEq}
\end{equation}

Finally, by employing (\ref{P1}), (\ref{P2}), (\ref{IneqLevy}) and (\ref%
{IntEq}) in connection with dominated convergence, we get that $u\in
C^{1}([0,T],C_{b}(\mathbb{R}^{d})).$
\end{proof}

\bigskip

\begin{remark}
We mention that the statements of Theorem \ref{Theorem 5} and \ref{Theorem 7}
are valid for all dimensions $d\geq 1.$ The case $d=1$ can be shown by using
the inequality%
\begin{equation*}
\left\Vert \widehat{u}\right\Vert _{L^{1}(\mathbb{R})}\leq C\left\Vert
u\right\Vert _{H^{s}(\mathbb{R})}
\end{equation*}%
for $s>\frac{1}{2},$ where $\widehat{u}$ is the Fourier transform of $u\in
H^{s}(\mathbb{R})=W^{s,2}(\mathbb{R})$ and $C$ a universal constant. See 
\cite{AF03} and Section 2.4.
\end{remark}

\bigskip

Using Theorem \ref{Theorem 7} we can rewrite $\int_{0}^{t}b(s,X_{s})ds$ in (%
\ref{eq:main}) in terms of a more regular expression.

\begin{corollary}[Representation of $\protect\int_{0}^{t}b(s,X_{s})ds$ ]
\label{Corollary 8} Retain the assumptions of Theorem \ref{Theorem 7} for $%
\phi =-b$ in (\ref{eq:main}). Suppose the drift coefficient $b$ admits the
existence of a strong solution $X_{\cdot }$ to (\ref{eq:main}). Then we have
the following representation:%
\begin{eqnarray*}
&&\int_{0}^{t}b(s,X(s))ds \\
&=&u(0,x)-u(t,X(t))+\int_{0}^{t}\int_{\mathbb{R}^{d}}\{u(s,X(s^{-})+\gamma
(z))-u(s,X(s^{-}))\}\widetilde{N}(ds,dz),
\end{eqnarray*}%
where%
\begin{equation*}
\gamma (z):=\mathbf{1}_{\{\left\vert z\right\vert \leq 1\}}z.
\end{equation*}
\end{corollary}

\begin{proof}
Let $u$ be the solution to the backward Kolmogorov equation in Theorem \ref%
{Theorem 7} for $\phi =-b$. Then, using It\^{o}'s Lemma, we obtain: 
\begin{align*}
& u(t,X(t))=u(0,x)+\int_{0}^{t}\frac{\partial u}{\partial s}%
(s,X(s))ds+\int_{0}^{t}b(s,X(s))\nabla _{x}u(s,X(s))ds \\
& +\int_{0}^{t}\int_{\mathbb{R}^{d}}\{u(s,X(t^{-})+\gamma
(z))-u(s,X(t^{-}))-\gamma (z)D_{x}u(s,X(s))\}\nu (dz)ds \\
& +\int_{0}^{t}\int_{\mathbb{R}}\{u(s,X(t^{-})+\gamma (z))-u(s,X(t^{-}))\}%
\widetilde{N}(ds,dz).
\end{align*}%
Because of (\ref{Generator}) and%
\begin{equation*}
\frac{\partial u}{\partial t}+b\nabla u+\mathcal{L}u=-b
\end{equation*}%
we get 
\begin{align*}
& \int_{0}^{t}b(s,X(s))ds \\
& =u(0,x)-u(t,X(t))+\int_{0}^{t}\int_{\mathbb{R}}\{u(s,X(t^{-})+\gamma
(z))-u(s,X(t^{-}))\}\widetilde{N}(ds,dz).
\end{align*}
\end{proof}

\bigskip

\section{Construction of Solutions to SDE's via the Compactness Criterion
for $L^{2}(\protect\mu )$}

\bigskip In this section we want to apply the compactness criterion in
Theorem \ref{Theorem 3} in connection with the results of the previous
section to construct strong solutions to the SDE (\ref{eq:main}), when $%
L_{t},0\leq t\leq T$ is a truncated $\alpha -$stable process of index $%
\alpha \in (1,2)$ and the drift coefficient $b\in C([0,T],C_{b}^{\beta }(%
\mathbb{R}^{d}))$ such that $\alpha +\beta >2$.

To this end, we aim defining a self adjoint operator $A$ on $L^{2}((0,\tau
))\otimes L^{2}(\nu )$ (for fixed $\tau >0$) which admits a compact inverse $%
A^{-1}:L^{2}((0,\tau ))\otimes L^{2}(\nu )\longrightarrow L^{2}((0,\tau
))\otimes L^{2}(\nu )$. More precisely, the operator $A$ is constructed as
follows:

Let the function $p$ (potential) be given by 
\begin{equation}
p(t,x)=\left\{ 
\begin{array}{cc}
\frac{1}{\left\vert x\right\vert ^{d+\alpha +\delta }} & ,\text{ if }%
\left\vert x\right\vert \leq \frac{1}{2} \\ 
\frac{2^{d+\alpha }}{(1-\left\vert x\right\vert )^{1/2}} & ,\text{ if }\frac{%
1}{2}<\left\vert x\right\vert <1%
\end{array}%
\right.  \label{p}
\end{equation}%
for some $\delta >0$ such that $\alpha +\delta <2$.

Consider now for fixed $\tau >0$ the symmetric form $\mathcal{E}$ on $%
L^{2}((0,\tau ))\otimes L^{2}(\nu )$ defined as%
\begin{eqnarray}
\mathcal{E}(f,g) &=&\int_{0}^{\tau }\int_{0}^{\tau }\int_{\left\vert
y\right\vert <1}\int_{\left\vert x\right\vert <1}\frac{\left(
f(t_{1},x)-f(t_{2},y)\right) \left( g(t_{1},x)-g(t_{2},y)\right) }{%
(\left\vert t_{1}-t_{2}\right\vert +\left\vert x-y\right\vert )^{d+1+2s}}%
dxdydt_{1}dt_{2}  \notag \\
&&+\int_{0}^{\tau }\int_{\left\vert x\right\vert <1}p(t,x)f(t,x)g(t,x)dxdt
\label{SymmForm}
\end{eqnarray}%
for functions $f,g$ in the (dense) domain $D(\mathcal{E})\subset
L^{2}((0,\tau ))\otimes L^{2}(\nu )$ and a fixed $s$ with%
\begin{equation*}
0<s<\frac{1}{2}\text{,}
\end{equation*}%
where%
\begin{eqnarray*}
&&D(\mathcal{E}) \\
&:&=\left\{ f:\right. \left\Vert f\right\Vert _{L^{2}((0,\tau ))\otimes
L^{2}(\nu )}^{2}+\int_{0}^{\tau }\int_{0}^{\tau }\int_{\left\vert
y\right\vert <1}\int_{\left\vert x\right\vert <1}\frac{\left\vert
f(t_{1},x)-f(t_{2},y)\right\vert ^{2}}{(\left\vert t_{1}-t_{2}\right\vert
+\left\vert x-y\right\vert )^{d+1+2s}}dxdydt_{1}dt_{2} \\
&&\left. +\int_{0}^{\tau }\int_{\left\vert x\right\vert <1}p(t,x)\left\vert
f(t,x)\right\vert ^{2}dxdt<\infty \right\} .
\end{eqnarray*}%
Then $\mathcal{E}$ is a positive symmetric closed form and we can find by
Kato%
\'{}%
s first representation theorem (see e.g. \cite{Kato95}) a positive self
adjoint operator $T_{\mathcal{E}}$ such that%
\begin{equation*}
\mathcal{E}(f,g)=\left( f,T_{\mathcal{E}}g\right) _{L^{2}((0,\tau ))\otimes
L^{2}(\nu )}
\end{equation*}%
for for all $g\in D(T_{\mathcal{E}}),$ $f\in D(\mathcal{E}).$ Further, since 
\begin{equation*}
\mathcal{E}(f,f)\geq \left\Vert f\right\Vert _{L^{2}((0,\tau ))\otimes
L^{2}(\nu )}^{2}
\end{equation*}%
for all $f\in D(\mathcal{E})$, that is the form $\mathcal{E}$ is bounded
from below by a positive number, we also have that $D(\mathcal{E})=D(T_{%
\mathcal{E}}^{1/2})$. See \cite{Kato95}.

Let us now define the operator $A$ in Theorem \ref{Theorem 3} as 
\begin{equation}
A=T_{\mathcal{E}}^{1/2}.  \label{A2}
\end{equation}%
We want to show that $A$ has a discrete spectrum and a compact operator
inverse $A^{-1}$. To verify this we prove that $T_{\mathcal{E}}$ has a
discrete spectrum with existing compact operator inverse. Before we proceed
we briefly introduce some notation: Consider now a general symmetric closed
form $\mathcal{E}$ bounded from below by a positive number with a domain $D(%
\mathcal{E})$ that is dense in the Hilbert space $H=L^{2}((0,\tau ))\otimes
L^{2}(\nu )$. We assume here that $\nu $ is a L\'{e}vy measure with
Lebesgue-density $w$. Let $\Omega $ be an open subset of $(0,\tau )\times 
\mathbb{R}^{d}$ and we assume that $\Omega $ is the union of an increasing
sequence of open sets $\Omega _{k}\subset \Omega ,k\geq 1.$ Further, we
denote by $H_{\mathcal{E}}(\Omega _{k})$ the inner product space with
respect to $\left\{ f\mathbf{1}_{\Omega _{k}}:f\in D(\mathcal{E})\right\} $
and the inner product $\left( f,g\right) _{\mathcal{E}}=\mathcal{E}(f,g)$.
Similarly, we define space $H_{\mathcal{E}}(\Omega )$.

We need the following auxiliary result:

\begin{lemma}
\label{Lewis} Let $\Omega $ be as above and assume that that the identity
maps $i_{k}:H_{\mathcal{E}}(\Omega _{k})\longrightarrow L^{2}(\Omega
_{k},dt\times \nu ),k\geq 1$ are compact. Suppose there is a positive-valued
function $p$ on $\Omega $ and a sequence $\varepsilon _{k},k\geq 1$ of
positive numbers with $\varepsilon _{k}\longrightarrow 0,k\longrightarrow
\infty $ such that%
\begin{equation}
w(x)/p(x)<\varepsilon _{k}  \label{L1}
\end{equation}%
for a.e. $x\in $ $\Omega -\Omega _{k}$ and%
\begin{equation}
\int_{\Omega -\Omega _{k}}p(t,x)\left\vert f(t,x)\right\vert ^{2}dx\leq 
\mathcal{E}(f,f)  \label{L2}
\end{equation}%
for all $f\in D(\mathcal{E})$. Then $T_{\mathcal{E}}$ has a discrete
spectrum and a compact inverse $T_{\mathcal{E}}^{-1}$.
\end{lemma}

\begin{proof}
\bigskip See Lemma 1 in \cite{Lew82}.
\end{proof}

\bigskip We now choose the function $p$ in Lemma \ref{Lewis} as in (\ref{p})
and we assume that $\nu $ is the L\'{e}vy measure of a truncated $\alpha -$%
stable L\'{e}vy process. Further, suppose that $\Omega _{k}\subset \Omega
:=(0,\tau )\times (U_{1}(0)-\{0\})$ with $\pi _{2}(\Omega -\Omega _{k})$ is
bounded away from $y=0$ and $\left\{ y:\left\vert y\right\vert =1\right\}
,k\geq 1$ ($\pi _{2}((t,y))=y$ projection onto the spatial component) such
that each $\Omega _{k}$ is of class $C^{0,1}$ with bounded boundary and $%
\Omega _{k}\nearrow \Omega $ and such that (\ref{L1}) is fulfilled. Then we
observe that $L^{2}(\Omega _{k},\lambda )$ ($\lambda $ Lebesgue measure on $%
\mathbb{R}^{d+1}$ ) and $L^{2}(\Omega _{k},dt\times \nu )$ coincide and that
their corresponding norms are equivalent for each $k$. So the latter, the
definition of $\mathcal{E}$ in (\ref{SymmForm}) in connection with (\ref{p})
and compactness results for fractional spaces $W^{s,p}(\Omega )$ (see e.g. 
\cite{AF03} or \cite{NPV}) imply that the identity maps $i_{k}:H_{\mathcal{E}%
}(\Omega _{k})\longrightarrow L^{2}(\Omega _{k},dt\times \nu ),k\geq 1$ are
compact. Finally, we also see that condition (\ref{L2}) is an immediate
consequence of the definition of $\mathcal{E}$. Hence, it follows from Lemma %
\ref{Lewis} that $T_{\mathcal{E}}$ has a discrete spectrum and a compact
inverse $T_{\mathcal{E}}^{-1}$. Using this we find that the operator $A$ in (%
\ref{A2}) satisfies the assumptions of Theorem \ref{Theorem 3}.

\bigskip

In order to apply Theorem \ref{Theorem 3} to the construction of solutions
to the SDE (\ref{eq:main}) we need the following estimate with respect to
the operator $A$ in (\ref{A2}):

\begin{lemma}
\label{Lemma 13}Let $b\in C([0,T],C_{b}^{\infty }(\mathbb{R}^{d}))$.
Further, let $X_{\cdot }$ be the unique strong solution to (\ref{eq:main})
with respect to the drift coefficient $b$. Then for sufficiently small $%
T<\infty $ we have that%
\begin{equation*}
E[\left\Vert AD_{\cdot ,\cdot }X(\tau )\right\Vert _{L^{2}((0,\tau ))\otimes
L^{2}(\nu )}^{2}]\leq K\exp (TMH_{1}(\left\Vert b\right\Vert _{C_{b}^{\beta
}}^{2}))
\end{equation*}%
for all $0<\tau \leq T$, where $K,M<\infty $ are constants being independent
of $b$ and where $H_{1}$ is a non-negative continuous function given by 
\begin{equation*}
H_{1}(y):=\frac{(y+1)^{2}}{(1-C^{2}(T)y)},0\leq y<\frac{1}{C^{2}(T)}\text{.}
\end{equation*}%
for a constant $C(T)$ with $C(T)\longrightarrow 0$ as $T\searrow 0$.
\end{lemma}

\begin{proof}
We know from Corollary \ref{Corollary 8} that we can rewrite the SDE (\ref%
{eq:main}) as%
\begin{eqnarray*}
X(t) &=&x+\int_{0}^{t}b(s,X(s))ds+L_{t}=u(0,x)-u(t,X(t)) \\
&&+\int_{0}^{t}\int_{\mathbb{R}^{d}}\{u(s,X(s^{-})+\gamma
(z))-u(s,X(s^{-}))\}\widetilde{N}(ds,dz) \\
&&+\int_{0}^{t}\int_{\mathbb{R}^{d}}\gamma (z)\widetilde{N}(ds,dz),
\end{eqnarray*}%
where $u\in C([0,T],C_{b}^{2}(\mathbb{R}^{d}))\cap C^{1}([0,T],C_{b}(\mathbb{%
R}^{d}))$ is the solution to the backward Kolmogorov equation (\ref%
{BackwardK}) in Theorem \ref{Theorem 7} and where%
\begin{equation*}
\gamma (z):=\mathbf{1}_{\{\left\vert z\right\vert \leq 1\}}z.
\end{equation*}%
So it follows from the properties of the Malliavin derivative $D$ associated
with our L\'{e}vy process (see e.g.\cite{NOP09}) that for all $0\leq l\leq t$
and $y$ 
\begin{eqnarray*}
D_{l,y}X(t) &=&u(t,X(t))-u(t,X(t)+D_{l,y}X(t)) \\
&&+\int_{0}^{t}\int_{\mathbb{R}^{d}}\{u(s,X(s^{-})+\gamma
(z)+D_{l,y}X(s^{-}))-u(s,X(s^{-})+\gamma (z)) \\
&&-(u(s,X(s^{-})+D_{l,y}X(s^{-}))-u(s,X(s^{-}))\}\widetilde{N}(ds,dz) \\
&&+u(l,X(l^{-})+\gamma (y))-u(l,X(l^{-}))+\gamma (y)
\end{eqnarray*}%
holds. Thus, using the mean value theorem we get%
\begin{eqnarray*}
&&D_{l_{1},y_{1}}X(t)-D_{l_{2},y_{2}}X(t) \\
&=&u(t,X(t)+D_{l_{2},y_{2}}X(t))-u(t,X(t)+D_{l_{1},y_{1}}X(t)) \\
&&+\int_{0}^{t}\int_{\mathbb{R}^{d}}\{u(s,X(s^{-})+\gamma
(z)+D_{l_{1},y_{1}}X(s^{-}))-u(s,X(s^{-})+D_{l_{1},y_{1}}X(s^{-})) \\
&&-u(s,X(s^{-})+\gamma
(z)+D_{l_{2},y_{2}}X(s^{-}))+u(s,X(s^{-})+D_{l_{2},y_{2}}X(s^{-}))\}%
\widetilde{N}(ds,dz) \\
&&+u(l_{1},X(l_{1}^{-})+\gamma (y_{1}))-u(l_{1},X(l_{1}^{-})) \\
&&-u(l_{2},X(l_{2}^{-})+\gamma (y_{2}))+u(l_{2},X(l_{2}^{-})) \\
&&+\gamma (y_{1})-\gamma (y_{2}) \\
&=&u(t,X(t)+D_{l_{2},y_{2}}X(t))-u(t,X(t)+D_{l_{1},y_{1}}X(t)) \\
&&+\int_{0}^{t}\int_{\mathbb{R}^{d}}\{%
\int_{0}^{1}(Du(s,X(s^{-})+D_{l_{1},y_{1}}X(s^{-})+\theta \gamma (z)) \\
&&-Du(s,X(s^{-})+D_{l_{2},y_{2}}X(s^{-})+\theta \gamma (z)))d\theta \gamma
(z)\}\widetilde{N}(ds,dz) \\
&&+u(l_{1},X(l_{1}^{-})+\gamma (y_{1}))-u(l_{1},X(l_{1}^{-})) \\
&&-u(l_{1},X(l_{2}^{-})+\gamma (y_{2}))+u(l_{2},X(l_{2}^{-})) \\
&&+\gamma (y_{1})-\gamma (y_{2}) \\
&=&\int_{0}^{1}Du(t,X(t)+\theta
(D_{l_{2},y_{2}}X(t)-D_{l_{1},y_{1}}X(t)))(D_{l_{2},y_{2}}X(t)-D_{l_{1},y_{1}}X(t))d\theta
\\
&&+\int_{0}^{t}\int_{\mathbb{R}^{d}}\{\int_{0}^{1}%
\int_{0}^{1}D^{2}u(s,X(s^{-})+\tau
(D_{l_{1},y_{1}}X(s^{-})-D_{l_{2},y_{2}}X(s^{-}))+\theta \gamma (z)) \\
&&[(D_{l_{1},y_{1}}X(s^{-})-D_{l_{2},y_{2}}X(s^{-})),\gamma (z)]d\tau
d\theta \}\widetilde{N}(ds,dz) \\
&&+u(l_{1},X(l_{1}^{-})+\gamma (y_{1}))-u(l_{1},X(l_{1}^{-})) \\
&&-u(l_{2},X(l_{2}^{-})+\gamma (y_{2}))+u(l_{2},X(l_{2}^{-})) \\
&&+\gamma (y_{1})-\gamma (y_{2}).
\end{eqnarray*}%
On the other hand, by repeated use of the mean value theorem we also have
that

\begin{eqnarray}
&&u(l_{1},X(l_{1}^{-})+\gamma (y_{1}))-u(l_{2},X(l_{2}^{-})+\gamma (y_{2})) 
\notag \\
&=&\int_{0}^{1}(\frac{\partial }{\partial t}u(l_{1}+\theta
(l_{1}-l_{2}),X(l_{1}^{-})+\gamma (y_{1})+\theta
(X(l_{1}^{-})-X(l_{2}^{-})+\gamma (y_{1})-\gamma (y_{2}))),  \notag \\
&&\frac{\partial }{\partial y}u(l_{1}+\theta
(l_{1}-l_{2}),X(l_{1}^{-})+\gamma (y_{1})+\theta
(X(l_{1}^{-})-X(l_{2}^{-})+\gamma (y_{1})-\gamma (y_{2}))))  \notag \\
&&\cdot ((l_{1}-l_{2}),(X(l_{1}^{-})-X(l_{2}^{-})+\gamma (y_{1})-\gamma
(y_{2}))^{T})d\theta .  \label{MV}
\end{eqnarray}%
Further, since%
\begin{equation*}
X(l_{1}^{-})-X(l_{2}^{-})=\int_{l_{1}}^{l_{2}}b(s,X(s))ds+%
\int_{l_{1}^{-}}^{l_{2}^{-}}\int_{\mathbb{R}^{d}}z\widetilde{N}(ds,dz)
\end{equation*}%
for $l_{1}\leq l_{2}$ It\^{o}'s isometry yields%
\begin{equation*}
E[\left\vert X(l_{1}^{-})-X(l_{2}^{-})\right\vert ^{2}]\leq C(\left\vert
l_{1}-l_{2}\right\vert ^{2}\left\Vert b\right\Vert _{\infty }^{2}+\left\vert
l_{1}-l_{2}\right\vert ).
\end{equation*}%
Using the latter, (\ref{MV}), the estimates (\ref{K1}), (\ref{K2}) in
Theorem \ref{Theorem 7} and (\ref{IneqLevy}) in connection with (\ref{IntEq}%
) we get that%
\begin{eqnarray*}
&&E[\left\vert u(l_{1},X(l_{1}^{-})+\gamma
(y_{1}))-u(l_{2},X(l_{2}^{-})+\gamma (y_{2}))\right\vert ^{2}] \\
&\leq &K\left\Vert b\right\Vert _{C_{b}^{\beta }}^{2}(\left\vert
l_{1}-l_{2}\right\vert ^{2}+\left\vert l_{1}-l_{2}\right\vert ^{2}\left\Vert
b\right\Vert _{\infty }^{2}+\left\vert l_{1}-l_{2}\right\vert +\left\vert
\gamma (y_{1})-\gamma (y_{2})\right\vert ^{2}) \\
&\leq &L\left\Vert b\right\Vert _{C_{b}^{\beta }}^{2}(1+\left\Vert
b\right\Vert _{C_{b}^{\beta }}^{2})(\left\vert l_{1}-l_{2}\right\vert
+\left\vert \gamma (y_{1})-\gamma (y_{2})\right\vert ).
\end{eqnarray*}%
In the same way we also obtain that%
\begin{equation*}
E[\left\vert u(l_{2},X(l_{2}^{-}))-u(l_{1},X(l_{1}^{-}))\right\vert
^{2}]\leq CL\left\Vert b\right\Vert _{C_{b}^{\beta }}^{2}(1+\left\Vert
b\right\Vert _{C_{b}^{\beta }}^{2})\left\vert l_{1}-l_{2}\right\vert .
\end{equation*}

By employing the It\^{o} isometry and once again the estimates (\ref{K1}), (%
\ref{K2}) for $T$ with $C(T)\left\Vert b\right\Vert _{C_{b}^{\beta }}\leq 
\frac{1}{2}$ in Theorem \ref{Theorem 7} we then find%
\begin{eqnarray*}
&&E[\left\vert D_{l_{1},y_{1}}X(t)-D_{l_{2},y_{2}}X(t)\right\vert
^{2}](1-C^{2}(T)\left\Vert b\right\Vert _{C_{b}^{\beta }}^{2}) \\
&\leq &K\{\left\Vert D^{2}u\right\Vert _{\infty }^{2}\int_{\mathbb{R}%
^{d}}\left\vert \gamma (z)\right\vert ^{2}\nu (dz)\int_{0}^{t}E[\left\vert
D_{l_{1},y_{1}}X(s^{-})-D_{l_{2},y_{2}}X(s^{-})\right\vert ^{2}]ds \\
&&+\left\Vert b\right\Vert _{C_{b}^{\beta }}^{2}(1+\left\Vert b\right\Vert
_{C_{b}^{\beta }}^{2})(\left\vert l_{1}-l_{2}\right\vert +\left\vert \gamma
(y_{1})-\gamma (y_{2})\right\vert )\}.
\end{eqnarray*}%
Hence%
\begin{eqnarray*}
&&E[\left\vert D_{l_{1},y_{1}}X(t)-D_{l_{2},y_{2}}X(t)\right\vert ^{2}] \\
&\leq &M\{\frac{\left\Vert b\right\Vert _{C_{b}^{\beta }}^{2}}{%
(1-C^{2}(T)\left\Vert b\right\Vert _{C_{b}^{\beta }}^{2})}%
\int_{0}^{t}E[\left\vert
D_{l_{1},y_{1}}X(s^{-})-D_{l_{2},y_{2}}X(s^{-})\right\vert ^{2}]ds \\
&&+\frac{\left\Vert b\right\Vert _{C_{b}^{\beta }}^{2}(\left\Vert
b\right\Vert _{C_{b}^{\beta }}^{2}+1)}{(1-C^{2}(T)\left\Vert b\right\Vert
_{C_{b}^{\beta }}^{2})}(\left\vert l_{1}-l_{2}\right\vert +\left\vert \gamma
(y_{1})-\gamma (y_{2})\right\vert )\} \\
&\leq &MH_{1}(\left\Vert b\right\Vert _{C_{b}^{\beta
}}^{2})\{\int_{0}^{t}E[\left\vert
D_{l_{1},y_{1}}X(s)-D_{l_{2},y_{2}}X(s)\right\vert ^{2}]ds \\
&&+(\left\vert l_{1}-l_{2}\right\vert +\left\vert \gamma (y_{1})-\gamma
(y_{2})\right\vert )\},
\end{eqnarray*}%
where%
\begin{equation*}
H_{1}(y):=\frac{(y+1)^{2}}{(1-C^{2}(T)y)},0\leq y<\frac{1}{C^{2}(T)}\text{.}
\end{equation*}%
Therefore we get%
\begin{eqnarray}
&&\int_{0}^{\tau }\int_{0}^{\tau }\int_{\left\vert y\right\vert
<1}\int_{\left\vert x\right\vert <1}\frac{E[\left\vert
D_{l_{1},y_{1}}X(t)-D_{l_{2},y_{2}}X(t)\right\vert ^{2}]}{(\left\vert
l_{1}-l_{2}\right\vert +\left\vert y_{1}-y_{2}\right\vert )^{d+1+2s}}%
dy_{1}dy_{2}dl_{1}dl_{2}  \notag \\
&\leq &MH_{1}(\left\Vert b\right\Vert _{C_{b}^{\beta
}}^{2})\{\int_{0}^{t}\int_{0}^{\tau }\int_{0}^{\tau }\int_{\left\vert
y\right\vert <1}\int_{\left\vert x\right\vert <1}\frac{E[\left\vert
D_{l_{1},y_{1}}X(s)-D_{l_{2},y_{2}}X(s)\right\vert ^{2}]}{(\left\vert
l_{1}-l_{2}\right\vert +\left\vert y_{1}-y_{2}\right\vert )^{d+1+2s}}%
dy_{1}dy_{2}dl_{1}dl_{2}ds  \notag \\
&&+\int_{0}^{\tau }\int_{0}^{\tau }\int_{\left\vert y\right\vert
<1}\int_{\left\vert x\right\vert <1}\frac{\left\vert l_{1}-l_{2}\right\vert
+\left\vert \gamma (y_{1})-\gamma (y_{2})\right\vert }{(\left\vert
l_{1}-l_{2}\right\vert +\left\vert y_{1}-y_{2}\right\vert )^{d+1+2s}}%
dy_{1}dy_{2}dl_{1}dl_{2}\}.  \label{E1}
\end{eqnarray}%
Similarly, we find%
\begin{eqnarray*}
&&E[\left\vert D_{l,y}X(t)\right\vert ^{2}] \\
&\leq &MH_{1}(\left\Vert b\right\Vert _{C_{b}^{\beta
}}^{2})\{\int_{0}^{t}E[\left\vert D_{l,y}X(s)\right\vert ^{2}]ds+\left\vert
\gamma (y)\right\vert ^{2}\}.
\end{eqnarray*}%
So 
\begin{eqnarray}
&&\int_{0}^{\tau }\int_{\left\vert y\right\vert <1}p(l,y)E[\left\vert
D_{l,y}X(t)\right\vert ^{2}]dydl  \notag \\
&\leq &MH_{1}(\left\Vert b\right\Vert _{C_{b}^{\beta
}}^{2})\{\int_{0}^{t}\int_{0}^{\tau }\int_{\left\vert y\right\vert
<1}p(l,y)E[\left\vert D_{l,y}X(s)\right\vert ^{2}]dydlds  \notag \\
&&+\int_{0}^{\tau }\int_{\left\vert y\right\vert <1}p(l,y)\left\vert \gamma
(y)\right\vert ^{2}]dydl\},  \label{E2}
\end{eqnarray}%
where the potential $p$ is defined as in (\ref{p}). By combining the
estimates (\ref{E1}) and (\ref{E2}) we obtain%
\begin{eqnarray*}
&&E[\left\Vert AD_{\cdot ,\cdot }X(t)\right\Vert _{L^{2}((0,\tau ))\otimes
L^{2}(\nu )}^{2}] \\
&\leq &MH_{1}(\left\Vert b\right\Vert _{C_{b}^{\beta
}}^{2})\{\int_{0}^{t}E[\left\Vert AD_{\cdot ,\cdot }X(s)\right\Vert
_{L^{2}((0,\tau ))\otimes L^{2}(\nu )}^{2}]ds \\
&&+K\},
\end{eqnarray*}%
where 
\begin{eqnarray*}
K &:&=\int_{0}^{\tau }\int_{0}^{\tau }\int_{\left\vert y\right\vert
<1}\int_{\left\vert x\right\vert <1}\frac{1}{(\left\vert
l_{1}-l_{2}\right\vert +\left\vert y_{1}-y_{2}\right\vert )^{d+2s}}%
dy_{1}dy_{2}dl_{1}dl_{2} \\
&&+\int_{0}^{\tau }\int_{\left\vert y\right\vert <1}p(l,y)\left\vert \gamma
(y)\right\vert ^{2}]dydl \\
&<&\infty ,
\end{eqnarray*}%
since $0<s<\frac{1}{2}.$ By Picard iteration one verifies that $E[\left\Vert
AD_{\cdot ,\cdot }X(t)\right\Vert _{L^{2}((0,\tau ))\otimes L^{2}(\nu
)}^{2}]<\infty .$

So we can apply Gronwall's Lemma and get%
\begin{equation*}
E[\left\Vert AD_{\cdot ,\cdot }X(t)\right\Vert _{L^{2}((0,\tau ))\otimes
L^{2}(\nu )}^{2}]\leq K\exp (TMH_{1}(\left\Vert b\right\Vert _{C_{b}^{\beta
}}^{2})).
\end{equation*}%
Thus the proof follows.
\end{proof}

\bigskip

We also want to employ the following result, whose proof can be found in 
\cite{Mitoma}:

\begin{theorem}
\label{DCompactness}Let $\Phi ^{\shortmid }$ be the topological dual of a
countably Hilbertian nuclear space $\Phi $. Further, consider the Skorohod
space $D([0,T],\Phi ^{\shortmid })$ of functions $f:[0,T]\longrightarrow
\Phi ^{\shortmid }$ with right-continuous paths and existing left limits.
Then a set $A\subset D([0,T],\Phi ^{\shortmid })$ is relatively compact if
and only if the set $\{f(\cdot )[\phi ]:f\in A\}$ is relatively compact in $%
D([0,T],\mathbb{R})$ for all $\phi \in \Phi $.
\end{theorem}

\bigskip

Let us now consider a function $b\in C([0,T];C_{b}^{\beta }(\mathbb{R}^{d}))$%
. Then we know from the proof of Theorem \ref{Theorem 5} that there exists $%
b_{n}\in C([0,T],C_{b}^{\infty }(\mathbb{R}^{d})),n\geq 1$ such that%
\begin{equation}
\left\Vert b_{n}\right\Vert _{C_{b}^{\beta }}\leq \left\Vert b\right\Vert
_{C_{b}^{\beta }}  \label{Krylov1}
\end{equation}%
for all $n$. Further, we have that 
\begin{equation}
b_{n_{k}(t)}(t,\cdot )\longrightarrow b(t,\cdot )\text{ in }C^{\delta }(K)
\label{Krylov2}
\end{equation}%
for all $t$, any compact set $K\subset \mathbb{R}^{d}$ and $0<\delta <\beta $
for a subsequence $n_{k}(t),k\geq 1$ depending on $t$ and $K$. See also p.
37 in \cite{Krylov}.

\bigskip

\begin{lemma}
\label{Mitoma}Suppose that $X_{t}^{n},0\leq t\leq T,n\geq 1$ are the unique
strong solutions to (\ref{eq:main}) with respect to the drift coefficients $%
b_{n}$ in (\ref{Krylov1}). Then there exists a subsequence $(n_{k})_{k\geq
1} $ which only depends on (a sufficiently small) $T$ such that for all $%
0\leq t\leq T:$ $X_{t}^{n_{k}}$ converges in $L^{2}(\Omega )$ for $%
k\longrightarrow \infty .$
\end{lemma}

\begin{proof}
We know that%
\begin{equation*}
X_{t}^{n}=x+\int_{0}^{t}b_{n}(s,X_{s})ds+L_{t},0\leq t\leq T.
\end{equation*}%
Let $\delta >0$ and consider a finite partition 
\begin{equation}
0=t_{0}<t_{1}<....<t_{n}=T  \label{Part}
\end{equation}
with $t_{j}-t_{j-1}\geq \delta $ for all $j=1,...,n.$

Then we have%
\begin{eqnarray}
&&X_{t_{j}}^{n}-X_{t_{j-1}}^{n}  \notag \\
&=&\int_{t_{j-1}}^{t_{j}}b_{n}(s,X_{s}^{n})ds+L_{t_{j}}-L_{t_{j-1}}
\label{i}
\end{eqnarray}%
for all $j=1,...,n..$

Now let $f$ be an element of the L\'{e}vy-Hida test function space $(%
\mathcal{S})\subset L^{2}(\Omega ).$ Denote by $(\mathcal{S})^{\ast }$ its
topological dual (L\'{e}vy-Hida distribution space). See e.g. \cite{NOP09}
and the references therein for further information on these spaces. Then $%
\left\langle (X_{t_{1}}^{n,i}-X_{t_{2}}^{n,i}),f\right\rangle _{(\mathcal{S}%
)^{\ast },(\mathcal{S})}=E[(X_{t_{1}}^{n,i}-X_{t_{2}}^{n,i})f],$ where $%
\left\langle \cdot ,\cdot \right\rangle _{(\mathcal{S})^{\ast },(\mathcal{S}%
)}$ is the dual pairing. So using (\ref{i}) we get%
\begin{eqnarray*}
&&E[(X_{t_{1}}^{n,i}-X_{t_{2}}^{n,i})f] \\
&=&%
\int_{t_{j-1}}^{t_{j}}E[b_{n}^{(i)}(s,X_{s}^{n})f]ds+E[(L_{t_{j}}-L_{t_{j-1}})f]
\end{eqnarray*}%
for all $j.$ Thus we it follows form H\"{o}lder's inequality and It\^{o}'s
isometry that%
\begin{equation*}
\left\vert E[(X_{t_{1}}^{n,i}-X_{t_{2}}^{n,i})f]\right\vert \leq C\left\vert
t_{j}-t_{j-1}\right\vert (E[f^{2}])^{1/2}
\end{equation*}%
for all $i$ and $j$ and a constant $C$ depending on $\left\Vert b\right\Vert
_{C_{b}^{\beta }}$ and the L\'{e}vy measure $\nu $. So 
\begin{equation*}
\sup_{n\geq 1}\omega ^{T}(\left\langle (X_{\cdot }^{n,i},f\right\rangle _{(%
\mathcal{S})^{\ast },(\mathcal{S})},\delta )\longrightarrow 0\text{ for }%
\delta \searrow 0,
\end{equation*}%
where $\omega ^{T}$ is the modulus given by%
\begin{equation*}
\omega ^{T}(g,\delta ):=\inf_{\{t_{j}\}}\max_{1\leq j\leq n}\sup
\{\left\vert g(t)-g(s)\right\vert :s,t\in \lbrack t_{j-1},t_{j})\},
\end{equation*}%
where the infimum is taken over partitions $\{t_{j}\}$ of the form (\cite%
{Xiong}).

So $\left\langle (X_{\cdot }^{n,i},f\right\rangle _{(\mathcal{S})^{\ast },(%
\mathcal{S})}$ is relatively compact in $D([0,T],\mathbb{R})$ for all $f\in (%
\mathcal{S})$. Since $(\mathcal{S})^{\ast }$ is the dual of a countably
Hilbertian nuclear space $(\mathcal{S})$, we can apply Theorem \ref%
{DCompactness} and find that there exists for all $i$ a subsequence $%
(n_{k}^{i})_{k\geq 1}$ which only depends on (a sufficiently small) $T$ such
that $X_{\cdot }^{n_{k}^{i},i}$ converges in $D([0,T];(\mathcal{S})^{\ast
}). $

On the other hand it follows from Lemma \ref{Lemma 13} and (\ref{Krylov1})
that for sufficiently small $T<\infty $ we have%
\begin{equation*}
E[\left\Vert AD_{\cdot ,\cdot }X^{n}(\tau )\right\Vert _{L^{2}((0,\tau
))\otimes L^{2}(\nu )}^{2}]\leq K\exp (TMH_{1}(\left\Vert b\right\Vert
_{C_{b}^{\beta }}^{2}))
\end{equation*}%
for all $0<\tau \leq T$, where $K,M<\infty $ are constants being independent
of $b$ and where $H_{1}$ is a non-negative continuous function on some
interval $[0,M]$ with $0\leq \left\Vert b\right\Vert _{C_{b}^{\beta }}^{2}<M$%
.

Then, applying Theorem \ref{Theorem 3} to the sequence $X_{t}^{n_{k}^{i},i}$
we find that for all $t$ and $i$ there exists a subsequence $%
m_{l}=m_{l}^{t,i},l\geq 1$ of $n_{k}^{i},k\geq 1$ and a $\widetilde{X}%
_{t}^{i}\in L^{2}(\Omega )$ such that%
\begin{equation}
X_{t}^{n_{m_{l}}^{i},i}\longrightarrow \widetilde{X}_{t}^{i}\text{ for }%
l\longrightarrow \infty  \label{Sub}
\end{equation}%
in $L^{2}(\Omega ).$

Let us show that 
\begin{equation*}
X_{t}^{n_{k}^{i},i}\longrightarrow \widetilde{X}_{t}^{i}\text{ for }%
k\longrightarrow \infty \text{ in }L^{2}(\Omega )
\end{equation*}%
for all $t,i.$ To this end we argue by contradiction. Assume that there
exists for some $t,i$ a $\varepsilon >0$ and a subsequence $\varphi
_{l},l\geq 1$ such that%
\begin{equation*}
\left\Vert X_{t}^{n_{\varphi _{l}}^{i},i}-\widetilde{X}_{t}^{i}\right\Vert
_{L^{2}(\Omega )}\geq \varepsilon .
\end{equation*}%
On the other hand we know by Theorem \ref{Theorem 3} that there exists a
subsequence $\phi _{r},r\geq 1$ of such that 
\begin{equation*}
X_{t}^{n_{\varphi _{\phi _{r}}}^{i},i}\longrightarrow \widetilde{Y}_{t}^{i}%
\text{ for }r\longrightarrow \infty \text{ in }L^{2}(\Omega ).
\end{equation*}%
But since%
\begin{equation*}
X_{t}^{n_{k}^{i},i}\longrightarrow \widetilde{X}_{t}^{i}\text{ for }%
k\longrightarrow \infty \text{ in }(\mathcal{S})^{\ast }
\end{equation*}%
because of (\ref{Sub}), we see that%
\begin{equation*}
\widetilde{Y}_{t}^{i}=\widetilde{X}_{t}^{i}.
\end{equation*}%
But this leads to the contradiction 
\begin{equation*}
\left\Vert X_{t}^{n_{\varphi _{\phi _{r}}}^{i},i}-\widetilde{X}%
_{t}^{i}\right\Vert _{L^{2}(\Omega )}\geq \varepsilon .
\end{equation*}%
So the proof follows.
\end{proof}

\bigskip

We are coming to the main result of our paper on SDE's with time-homogeneous
drift coefficients:

\begin{theorem}
\bigskip \label{MainResult}\bigskip Suppose that $L_{t},0\leq t\leq T$ is a $%
d$-dimensional truncated $\alpha -$stable process for $\alpha \in (1,2)$ and 
$d\geq 2$. Require that $b\in C_{b}^{\beta }(\mathbb{R}^{d})$ for $\beta \in
(0,1)$ such that $\alpha +\beta >2$. Then there exists for sufficiently
small $T>0$ a unique strong solution $X_{\cdot }$ to the SDE%
\begin{equation}
dX_{t}=b(X_{t})dt+dL_{t},0\leq t\leq T,X_{0}=x.  \label{SDE2}
\end{equation}%
Moreover, $X_{t}$ is Malliavin differentiable for all $0\leq t\leq T$.
\end{theorem}

\begin{proof}
\textbf{1. }\emph{Existence}: By (\ref{Krylov2} ) (see also the proof of
Theorem \ref{Theorem 5}) we find a subsequence $n_{k}^{\ast },k\geq 1$ such
that%
\begin{equation*}
b_{n_{k}^{\ast }}(y)\longrightarrow b(y)\text{ as }k\longrightarrow \infty
\end{equation*}%
for all $y$. Consider now the sequence of unique strong solutions $X_{\cdot
}^{k}$ to 
\begin{equation}
X_{t}^{k}=x+\int_{0}^{t}b_{n_{k}^{\ast }}(X_{s}^{k})ds+L_{t}  \label{SDE3}
\end{equation}%
with respect to the drift coefficients $b_{n_{k}^{\ast }},k\geq 1$ in (\ref%
{Krylov1}). Then we know from Lemma \ref{Mitoma} that there exists a
subsequence $(n_{k})_{k\geq 1}$ which only depends on (a sufficiently small) 
$T$ such that for all $0\leq t\leq T:$%
\begin{equation*}
X_{t}^{n_{k}}\longrightarrow X_{t}\text{ in }L^{2}(\Omega )
\end{equation*}%
for $k\longrightarrow \infty .$ On the other hand we obtain by dominated
convergence that$_{{}}$%
\begin{eqnarray*}
&&E[(\int_{0}^{t}b_{n_{n_{k}}^{\ast
}}(X_{s}^{n_{k}})ds-\int_{0}^{t}b(X_{s})ds)^{2}] \\
&=&E[(\int_{0}^{t}(b_{n_{n_{k}}^{\ast }}(X_{s}^{n_{k}})-b_{n_{n_{k}}^{\ast
}}(X_{s})+b_{n_{n_{k}}^{\ast }}(X_{s})-b(X_{s}))ds)^{2}] \\
&\leq &C(E[\int_{0}^{t}(b_{n_{n_{k}}^{\ast
}}(X_{s}^{n_{k}})-b_{n_{n_{k}}^{\ast
}}(X_{s}))^{2}ds]+E[\int_{0}^{t}(b_{n_{n_{k}}^{\ast
}}(X_{s})-b(X_{s}))^{2}ds]) \\
&\leq &C\left\Vert b\right\Vert _{C_{b}^{\beta
}}^{2}(E[\int_{0}^{t}\left\vert X_{s}^{n_{k}}-X_{s}\right\vert ^{2\beta
}ds]+E[\int_{0}^{t}(b_{n_{n_{k}}^{\ast }}(X_{s})-b(X_{s}))^{2}ds]) \\
&\leq &C\left\Vert b\right\Vert _{C_{b}^{\beta
}}^{2}(\int_{0}^{t}(E[\int_{0}^{t}\left\vert X_{s}^{n_{k}}-X_{s}\right\vert
^{2}])^{\beta }ds+E[\int_{0}^{t}(b_{n_{n_{k}}^{\ast
}}(X_{s})-b(X_{s}))^{2}ds]) \\
&\longrightarrow &0\text{ as }k\longrightarrow \infty .
\end{eqnarray*}%
So by passing to the limit in $L^{2}(\Omega )$ on both sides of (\ref{SDE3})
we get%
\begin{equation*}
X_{t}=x+\int_{0}^{t}b(X_{s})ds+L_{t},0\leq t\leq T\text{.}
\end{equation*}

\textbf{2. }\emph{Uniqueness}: Suppose that there are two solutions $%
X_{\cdot }^{1}$ and $X_{\cdot }^{2}$ to (\ref{SDE3}). Then it follows from
Corollary \ref{Corollary 8} and the mean value theorem that%
\begin{eqnarray*}
&&X^{1}(t)-X^{2}(t) \\
&=&\int_{0}^{t}(b(X^{1}(s))-b(X^{2}(s)))ds=u(t,X^{2}(t))-u(t,X^{1}(t)) \\
&&+\int_{0}^{t}\int_{\mathbb{R}^{d}}\{u(s,X^{1}(s^{-})+\gamma
(z))-u(s,X^{1}(s^{-})) \\
&&-u(s,X^{2}(s^{-})+\gamma (z))+u(s,X^{2}(s^{-}))\}\widetilde{N}(ds,dz) \\
&=&u(t,X^{2}(t))-u(t,X^{1}(t)) \\
&&+\int_{0}^{t}\int_{\mathbb{R}^{d}}\int_{0}^{1}\int_{0}^{1}D^{2}u(s,\theta
\gamma (z)+\tau (X^{1}(s^{-})-X^{2}(s^{-}))) \\
&&\lbrack (X^{1}(s^{-})-X^{2}(s^{-})),\gamma (z)]d\theta d\tau \widetilde{N}%
(ds,dz).
\end{eqnarray*}%
Using the It\^{o} isometry and the estimates (\ref{K1}), (\ref{K2}) we obtain%
\begin{eqnarray*}
&&E[\left\vert X^{1}(t)-X^{2}(t)\right\vert ^{2}] \\
&\leq &K\frac{\left\Vert b\right\Vert _{C_{b}^{\beta }}^{2}}{%
(1-C^{2}(T)\left\Vert b\right\Vert _{C_{b}^{\beta }}^{2})}%
\int_{0}^{t}E[\left\vert X^{1}(s)-X^{2}(s)\right\vert ^{2}]ds.
\end{eqnarray*}%
Hence Gronwall's Lemma gives%
\begin{equation*}
X_{\cdot }^{1}=X_{\cdot }^{2}.
\end{equation*}

The Malliavin differentiability of $X_{t}$ is a consequence of the fact (see
Lemma \ref{Lemma 13}) that%
\begin{eqnarray*}
&&E[\left\Vert D_{\cdot ,\cdot }X^{n_{k}}(\tau )\right\Vert _{L^{2}((0,\tau
))\otimes L^{2}(\nu )}^{2}] \\
&\leq &CE[\left\Vert AD_{\cdot ,\cdot }X^{n_{k}}(\tau )\right\Vert
_{L^{2}((0,\tau ))\otimes L^{2}(\nu )}^{2}] \\
&\leq &K\exp (TMH_{1}(\left\Vert b\right\Vert _{C_{b}^{\beta }}^{2})), \\
k &\geq &1,0<\tau \leq T
\end{eqnarray*}
and Lemma 1.2.3 in \cite{Nua06}.
\end{proof}

\bigskip

\begin{remark}
The proof of Theorem \ref{MainResult} and the preceding results which are
formulated with respect to time-inhomogeneous coefficients $b$ show that we
may choose in Theorem \ref{MainResult} drift coefficients of the form%
\begin{equation*}
b(t,x)=\sum_{i=1}^{m}f_{i}(t)b_{i}(x),
\end{equation*}%
where $f_{i},i=1,...,m$ are continuous functions and $b_{i}\in C_{b}^{\beta
}(\mathbb{R}^{d}),i=1,...,m.$
\end{remark}

\end{document}